\documentclass[aop,preprint]{imsart}

\RequirePackage{amsthm,amsmath,amsfonts,amssymb}
\RequirePackage[numbers]{natbib}
\RequirePackage[colorlinks,citecolor=blue,urlcolor=blue]{hyperref}
\RequirePackage{graphicx}

\RequirePackage{mathrsfs}
\usepackage{mathrsfs}

\startlocaldefs

\theoremstyle{plain}

\newtheorem{theorem}{Theorem}[section]
\newtheorem{lemma}[theorem]{Lemma}

\theoremstyle{remark}
\newtheorem{definition}[theorem]{Definition}

\allowdisplaybreaks[3]
\usepackage{color}
\newcommand{\bbA}{{\bf A}}
\newcommand{\bbP}{{\bf P}}
\newcommand{\bbE}{{\bf E}}

\newcommand{\bbC}{{\bf C}}
\newcommand{\bbD}{{\bf D}}
\newcommand{\bbS}{{\bf S}}
\newcommand{\bbR}{{\bf R}}
\newcommand{\bbT}{{\bf T}}
\newcommand{\bbQ}{{\bf Q}}
\newcommand{\bbI}{{\bf I}}
\newcommand{\bbM}{{\bf M}}

\newcommand{\bbK}{{\bf K}}
\newcommand{\bbB}{{\bf B}}
\newcommand{\bbX}{{\bf X}}
\newcommand{\bbx}{{\bf x}}

\newcommand{\bbY}{{\bf Y}}

\newcommand{\bbr}{{\bf r}}

\newcommand{\bbu}{{\bf u}}
\newcommand{\bbv}{{\bf v}}

\newcommand{\bbalp}{{\bf \alpha}}

\newcommand{\bqn}{\begin{eqnarray*}}
	\newcommand{\eqn}{\end{eqnarray*}}

\newcommand{\rE}{{\textrm{E}}}

\newcommand{\rP}{\textrm{P}}

\newcommand{\bqa}{\begin{eqnarray}}
	\newcommand{\eqa}{\end{eqnarray}}

\newtheorem{thm}{Theorem}[section]
\newtheorem{cor}[thm]{Corollary}

\newtheorem{remark}[thm]{Remark}

\endlocaldefs

\begin{document}
	
	\begin{frontmatter}
		\title{No Eigenvalues Outside the Support of the Limiting Spectral Distribution of Large Dimensional noncentral Sample Covariance Matrices}
		\runtitle{}
		
		\begin{aug}
			
			\author[A]{\fnms{Zhidong}~\snm{Bai}\ead[label=e1]{baizd@nenu.edu.cn}}
			\author[A]{\fnms{Jiang}~\snm{Hu}\ead[label=e2]{huj156@nenu.edu.cn}}
			\author[B]{\fnms{Jack W.}~\snm{Silverstein}\ead[label=e3]{jack@ncsu.edu}}
			\author[A]{\fnms{Huanchao}~\snm{Zhou}\ead[label=e4]{zhouhc782@nenu.edu.cn}}

			\address[A]{KLASMOE and School of Mathematics and Statistics, Northeast Normal University, China \printead[presep={,\ }]{e1,e2,e4}}
			\address[B]{Department of Mathematics, North Carolina State University, USA \printead[presep={,\ }]{e3}}
			
		\end{aug}
		
		\begin{abstract}
			
			Let $ \bbB_n =\frac{1}{n}(\bbR_n + \bbT^{1/2}_n \bbX_n)(\bbR_n + \bbT^{1/2}_n \bbX_n)^* $, where $ \bbX_n $ is a  $ p \times n $ matrix with independent standardized random variables, 
			$ \bbR_n $ is a $ p \times n $ non-random matrix and 
			$ \bbT_{n} $ is a $ p \times p $ non-random, nonnegative definite Hermitian matrix.  The matrix $\bbB_n$ is referred to as the  information-plus-noise type matrix, where $\bbR_n$ contains the information and $\bbT^{1/2}_n \bbX_n$ is the noise matrix with the covariance matrix $\bbT_{n} $.
			It is known that, as $ n \to \infty $, if $ p/n $ converges to a positive number, the empirical spectral distribution of $ \bbB_n  $ converges almost surely to a nonrandom limit, under some mild conditions. 
			In this paper, we prove that, under certain conditions on the eigenvalues of $ \bbR_n $ and 
			$ \bbT_n $, for any closed interval outside the support of the limit spectral distribution, with probability one there will be no eigenvalues falling in this interval for all $ n $ sufficiently large.
		\end{abstract}
		
	
		
	\end{frontmatter}
	
	\section{Introduction}
	Let $ \bbX_n=(x_{ij}) $ be a  $ p \times n $ matrix of independent and 
	standardized random variables $ (\bbE x_{ij} =0, \bbE \lvert x_{ij} \rvert^2=1  ) $, $ \bbR_n $ be a $ p \times n $ non-random matrix and $ \bbT_{n} $ be a $ p \times p $ non-random nonnegative definite Hermitian matrix. 
	The matrix
	\begin{equation*} 
	\bbB_n =\frac{1}{n}(\bbR_n + \bbT^{1/2}_n \bbX_n)(\bbR_n + \bbT^{1/2}_n \bbX_n)^*,
	\end{equation*} 
	is referred to as the information-plus-noise type matrix, where the information is contained in the matrix 
	$ \bbR_n \bbR_n^* /n$ and the matrix $ \bbT^{1/2}_n \bbX_n $ is the additive noise.
	The limiting spectral distribution (LSD) of $\bbB_n$ has been studied in \cite{zhou2022limiting}, and the result is expressed in terms of the empirical spectral distribution (ESD) function $F^{\bbB_n}$. More specifically,  assume that  $u_i$ and $t_i$ are the paired eigenvalues of $\bbR_n\bbR_n^*/n$ and $\bbT_n$ in their simultaneous spectral decomposition. it is shown in \cite{zhou2022limiting} that  if the formation $\bbR_n\bbR_n^*/n$ commutes with the noise covariance $\bbT_n$, and as $\min\{p, n\} \to \infty$, $y_n = p/n \to y > 0$,  the two-dimensional distribution function
	$ H_n(u,t) = p^{-1} \sum_{i=1}^{p} I(u_i \le u, t_i \le t) $ converges weakly to a nonrandom limit distribution $ H(u, t) $,   then with probability one, $ F^{\bbB_n} $ converges in distribution to $ F, $ a nonrandom probability distribution function, whose Stieltjes transform $ s = s_F(z) $ satisfies the equation system
	\begin{equation}
	\left\{
	\begin{aligned}
	s=\int\frac{\mathrm{d}H(u,t)}{\frac{u}{1+yg}-(1+yst)z+t(1-y)}, \\
	g=\int\frac{t\mathrm{d}H(u,t)}{\frac{u}{1+yg}-(1+yst)z+t(1-y)}. \label{712.1}
	\end{aligned}
	\right.
	\end{equation}	
	Moreover, for each $ z\in \mathbb{C}^{+} $, $ (s, g) $ is the unique solution to $ \eqref{712.1} $ in  $ \mathbb{C}^{+} $.	
	Here and in the sequel,  the Stieltjes transform of $ F $ is defined as 
	\begin{align*}
	\begin{split}
	s_{F}(z)=\int \dfrac{1}{\lambda-z} \mathrm{d}F(\lambda), \quad 
	z\in \mathbb{C}^{+}\equiv \{ z\in \mathbb{C}: \Im z >0 \}, 
	\end{split}
	\end{align*}
	and $ F $ can be obtained by the inversion formula 
	\begin{align}
	F(b)-F(a)=\frac{1}{\pi}\lim_{v\rightarrow 0^{+}}\int_{a}^{b}\Im s_{F}(x+iv)\mathrm{d} x, \label{1.3}
	\end{align}
	where $ a $, $ b $ are continuity points of $ F $. 
	
	The analytic properties of the Stieltjes transform of the LSD of $\bbB_n$ are studied in \cite{zhou2023analysis}.  
	Let $\underline{\bbB}_n=\frac{1}{n}(\bbR_n+\bbT^{\frac{1}{2}}_n \bbX_n)^*(\bbR_n +  \bbT^{\frac{1}{2}}_n \bbX_n) $. 
	The eigenvalues of the matrix $ \underline{\bbB}_n $ are the same as those of the matrix $ \bbB_n $ except $ \lvert n-p \rvert $ zero eigenvalues. 
	Therefore, their ESDs and Stieltjes transforms have  the  following relations
	\begin{gather*}
	F^{\underline{\bbB}_n}=\left(1-\frac{p}{n} \right)I_{[0,\infty]}+\frac{p}{n} F^{ \bbB_n},\\
	\underline{s}(z)=-\dfrac{1-y}{z}+ys(z)\mbox{~~and~~}
	\underline{g}(z)=-\dfrac{1}{z(1+yg(z))},   
	\end{gather*}
	where $ \underline{s}(z)$ is the Stieltjes transform of the LSD 
	$ F^{\underline{\bbB}_n} $ and $ \underline{s}_{n}(z)=s_{F^{\underline{\bbB}_n}}(z) $ and  
	$ \underline{g}(z)$ is the limit of  $ \underline{g}_n=\frac{1}{p}\mathrm{tr} \bbT_n \left(\underline{\bbB}_n-z\bbI \right)^{-1} $.
	Then the equations in \eqref{712.1} become
	\begin{equation}
	\left\{ 
	\begin{aligned}
	z=&-\dfrac{1-y}{\underline{s}}-\frac{y}{\underline{s}}\int \dfrac{ \mathrm{d} H(t,u)}{1 + u\underline{g}(z) + t\underline{s}(z)},
	\\
	z=&-\dfrac{1}{\underline{g}}+y\int \dfrac{t \mathrm{d} H(t,u)}{1 + u\underline{g}(z) + t\underline{s}(z)}. \label{713}
	\end{aligned}
	\right.
	\end{equation}
	It is shown that for all $ x \in \mathbb{R}^+ $, $ \lim_{z\in \mathbb{C}^+\to x } s_F (z) \equiv s(x) $ exists.  And continuous dependence of $ F $ on $ y $ and $ H $ is readily apparent from the inversion formula $ \eqref{1.3} $ and $ \eqref{713} $. Moreover, away from zero, $ F $ also has a continuous density.
	Moreover, the support of a distribution function $ F $ is the set of all points $ x $ satisfying 
	$ F(x + \varepsilon) - F(x - \varepsilon) > 0  $ for all $ \varepsilon> 0. $
	Let $ S_F $ and $ S_H $ denote the support of $ F $ and $ H $, respectively. 
	Clearly, by definition of $ F $ and $ H $, we have $ S_F \subset [0, \infty)  $ and  $ S_H \subset [0, \infty)$. 
	Then, on intervals outside the support of probability distribution function 
	$ F $, $ s_F(x) $ exists and is increasing. 
	
	The focus of this paper is on intervals of $ \mathbb{R}^+ \equiv \mathbb{R}-\{0\} $ lying outside the support of $ F $.
	We will prove that when $ n $ is large, with probability one, there are no sample eigenvalues of $\bbB_n$ falling into the limiting spectral gaps. 
	If $ \bbT=\sigma^2 \bbI $, \cite{bai2012no}  proved that for any closed interval contained in an open interval in 
	$ \mathbb{R}^+ $ outside the supports of the limiting distribution $F^{y_n,H_n}$, then, almost surely, no eigenvalues of $\bbB_n$ will appear in this interval for all $ n $ large. And \cite{capitaine2014exact} proved the exact separation theorem.

	
	In this paper, we prove that, under certain conditions on the eigenvalues of $ \bbR_n $ and 
	$ \bbT_n $, for any closed interval outside the support of the limit spectral distribution, with probability one there will be no eigenvalues falling in this interval for all $ n $ sufficiently large. Our main result of this paper is as follows.
	
	\begin{theorem} \label{th7221}
		Assume that
		\begin{enumerate}
			\item[(a)] $ [a, b] \subset (c, d) \subset S^c_{\underline{F}^{y_n,H_n} }  $, with $ c > 0 $ for all large $ n $;
			\item[(b)] The matrix $\bbX_n$ is the $p\times n$ upper-left conner of the double array of random variables $ x_{ij} $ having means zero, variances one, second moments zero if complex and  there  is a random variable 
			$ X $ with finite fourth moment such that for a constant $ K $ and for all 
			$ x > 0 $ 
			$$  \frac{1}{p} \sum_{i =1}^{p} \bbP(\lvert x_{ij} \rvert > x ) \leq K\bbP(\lvert X \rvert) > x) , $$
			and 
			$$  \frac{1}{n}  \sum_{j=1} ^{n}\bbP(\lvert x_{ij} \rvert > x ) \leq K\bbP(\lvert X \rvert) > x) ; $$
			
			\item[(c)] There exists a positive function $ \psi(x) \uparrow \infty $ as $ x \to \infty$, and $ M > 0  $ such that
			$$\max_{ij} \bbE \lvert x^2_{ij} \rvert \psi(\lvert x_{ij} \rvert) \leq M; $$
			\item[(d)] $ n = n(p) $ with $ y_n = {p}/{n} \to y > 0 $  as $ n \to \infty $;
			\item[(e)] For $ n = 1, 2, \dots, $  $\bbR_n $ is a $ p \times n $ nonrandom matrix with $\dfrac{1}{\sqrt{n} } \bbR_n  $ uniformly bounded in special norm for all $ n $;
			
			\item[(f)] The matrix $ \bbT_n $ is uniformly bounded in spectral norm and 
			$ \lambda_{-1} \leq K $ for some constant $ K $, and is also commutative with $ (1/n)\bbR_n\bbR^*_n$ and their joint spectral distribution $ H_n(u, t)  $ tends to a proper distribution $ H(u, t) $, where $ \lambda_{-1} =\int t^{-1} \mathrm{d} H(u,t) $. 
		\end{enumerate}
		
		Then, we have that
		\begin{align} \label{7442}
		\bbP(\text{no eigenvalues of } \bbB_n \ \text{appear in} \ [a, b] \ \text{for all large} \ n) = 1. 
		\end{align}
		
	\end{theorem}
	\begin{remark}
		As mentioned in \cite{bai1998no,bai2012no,couillet2011deterministic}, assumptions (b)-(c) allow for the $ x_{ij} $ to depart from merely being
		i.i.d.. After suitable truncation, centralization, and scaling of the $ x_{ij} $'s one can assume these variables to be uniformly bounded.
	\end{remark}
	
	The rest of this paper is organized as follows. Some preliminary results of proving the theorem are introduced in Section 2. The proof of Theorem \ref{th7221} is split into the proofs of the convergence of the random part and the convergence of the non-random part Sections 3 and 4, respectively. We complete the proof of Theorem \ref{th7221} in Section 5. Some technical lemmas are given in Section 6.

	\section{Preliminary Results}	
	In this section, we give some preliminary results for the proof of Theorem \ref{th7221}.
	Before the truncation and centralization steps, we give a lemma below that can simplify the assumptions on the matrix of $ \bbX_n $. 
	\begin{lemma}\label{lemma5.11}
		Assume that the entries of $ \{x_{ij} \} $ are a double array of independent complex random variables with mean zero, variance $ \sigma^2 $, and satisfy the assumptions (b) -- (e) of Theorem \ref{th7221}. 
		Let $ \bbX_n = (x_{ij}; i \leq p, j \leq n) $ be the $ p \times n $ matrix of the upper-left corner of the double array. 
		Then, with probability one, we have
		\begin{align*}
		-2\sqrt{y}\sigma^2 &\leq  {\lim\inf}_{n\to \infty} \lambda_{\min}(\bbS_n - \sigma^2(1 + y)\bbI_n)\\
		&\leq {\lim \inf}_{n\to \infty} \lambda_{\max}(\bbS_n -\sigma^2(1 + y)\bbI_n)\leq 2\sqrt{y}\sigma^2, 
		\end{align*}
		where $\bbS_n=n^{-1}\bbX_n\bbX_n^*$.
		
	\end{lemma}
	\begin{proof}
		The proof of Lemma \ref{lemma5.11} is exactly the same as that of Theorem 5.10  of \cite{bai2010spectral} after the truncation and centralization. Therefore, we only present the truncation and centralization here, and the detailed proof of the lemma is omitted.  
		
		Without loss of generality, we assume $\sigma=1$. We first
		truncate the $x$-variables. Since $\rE|X|^4<\infty$, we can
		select a sequence of slowly decreasing constants $\delta_n\to 0$ such
		that $\sqrt n\delta_n$ is increasing, $\delta_n^{-2}n^2{\rP}(|X|>\delta_n\sqrt{n})\to 0$
		and 
		\begin{equation} \sum_{k} 
		\delta_{2^k}^{-2}2^{2k}\rP(|X|\ge
		2^{k/2}\delta_{2^k})<\infty. \label{sminfinite} \end{equation} 
		Define
		$x_{ijn}=x_{ij}I(|x_{ij}|\le \delta_n\sqrt{n})$ and construct a
		matrix $\widehat \bbS_n$ with the same structure of $\bbS_n$ by
		replacing $x_{ij}$ with $x_{ijn}$. Then, we have \bqn
		&&{\rP}\left(\widehat \bbS_n\neq \bbS_n,\mbox{i.o.}\right)\\
		&\le &\lim_{M\to\infty}\sum_{k=M}^\infty
		\rP\left(\bigcup_{2^k<n\le 2^{k+1}} \bigcup_{i\le p, j\le n}
		(x_{ijn}\neq x_{ij})\right)\\
		&\le& \lim_{M\to\infty}\sum_{k=M}^\infty
		\rP\left(\bigcup_{2^k<n\le 2^{k+1}} \bigcup_{i,j\le 2^{k+1}}
		(|x_{ij}|\ge \delta_{2^k}2^{k/2})\right)\\
		&=& \lim_{M\to\infty}\sum_{k=M}^\infty
		\rP\left(\bigcup_{i,j\le 2^{k+1}}
		(|x_{ij}|\ge \delta_{2^k}2^{k/2})\right)\\
		&\le&\lim_{M\to\infty}K\sum_{k=M}^\infty
		2^{2k+2}\rP\left(|X|\ge \delta_{2^k}2^{k/2}\right)=0. \eqn
		This shows that the truncation doesn't affect the limits of extreme eigenvalues of $\bbS_n$. Next, let $\widetilde \bbS_n$ be the matrix constructed as $\bbS_n$ with $x_{ij}$ replaced by $\tilde x_{ij}=x_{ijn}-\rE x_{ijn}$.  By Theorem A.46 of \cite{bai2010spectral}, we have 
		\bqn
		&&\max_{1\le i\le p}|\lambda_i(\widehat\bbS_n)-\lambda_i(\widetilde \bbS_n)|^2\le \left\|\frac1{\sqrt n} (\rE x_{ijn})\right\|^2\\
		&\le& \sup_{|c_1|^2+\cdots+|c_p^2|=1}\frac1n\sum_{i=1}^p\sum_{j=1}^n\left(\sum_{i=1}^p|c_i|\rE|x_{ij}I(|x_{ij}|>\delta_n\sqrt{n})\right)^2\\
		&\le& \frac1n \sum_{i=1}^p\sum_{j=1}^n\left(\int_{\delta_n\sqrt{n}}^\infty P(|x_{ij}|>x)\mathrm{d}x \right)^2\\
		&\le& \delta_n^{-1}n^{1/2}\left(K\int_{\delta_n\sqrt{n}}^\infty \rP(|X|>x)\mathrm{d}x\right)\\
		&=& K\delta_n^{-3}n^{-1/2}\to 0,
		\eqn
		where we have used the fact that $$ \int_{\delta_n\sqrt{n}}^\infty P(|x_{ij}|>x)\mathrm{d}x \le \delta_n^{-1}n^{-1/2}\int_{0}^\infty xP(|x_{ij}|\ge x)\mathrm{d}x \le \delta_n^{-1}n^{-1/2}. $$ 
		This shows that the centralization doesn't affect the limits of extreme eigenvalues of $\bbS_n$.
		
		Therefore, when prove Lemma \ref{lemma5.11}, we may assume that 
		\bqa
		&(1)~&\rE(x_{ij})=0, \ \ \rE(|x_{ij}|^2)\le 1,\ \mbox{ and }\
		\rE(|x_{ij}|^2)\to 1. \label{addtnal} \\
		&(2)~& \sum_{i\mbox{(or)}j}\rE(|x_{ij}|^\ell)\le \begin{cases}bn(\delta_n\sqrt{n})^{\ell-3}&\ \mbox{ for all }
			\ell\ge 3\cr
			p\mbox{ or } n&\ \mbox{for }\ell=1,2.\cr\end{cases}\nonumber \eqa
		Here, the third conclusion of assertion (1) in (\ref{addtnal}) follows from 
		\bqn
		1-\rE|x_{ijn}^2|=\rE|x_{ij}^2|I(|x_{ij}|\ge\delta_n\sqrt{n})\le \psi^{-1}(\delta_n\sqrt{n})\rE|x_{ij}^2|\psi(|x_{ij}|)\to 0.
		\eqn
		As for conclusion (2) in (\ref{addtnal}), it needs assumption (b) of Theorem \ref{th7221}, \bqn
		&&\sum_{i=1}^p\rE(|x_{ij}|^\ell)=\sum_{i=1}^p\ell \int_0^{\delta_n\sqrt{n}}x_{\ell-1}{\rP}(|x_{ij}|>x)\mathrm{d}x\le Kp\ell\int_0^{\delta_n\sqrt{n}}x^{\ell-1}{\rP}(|X|>x)\mathrm{d}x\\
		&=&Kp\rE|X|^\ell(|X|\le \delta_n\sqrt{n}),
		\eqn
		and the routine approach.  Similar to the other assertion. Then Lemma \ref{lemma5.11} can be proved by the same lines as those of Theorem 5.10 of  \cite{bai2010spectral}, with noticing the sufficiency of Lemma B.25 of  \cite{bai2010spectral} in the latter remaining hold after truncation and centralization.
	\end{proof}
	
	Now, we turn to the preliminary for the proof of Theorem \ref{th7221}. First, we define $ \hat{x}_{ij}=x_{ij} I (\lvert x_{ij}\rvert <C) -\bbE x_{ij}  I (\lvert x_{ij}\rvert <C) $   and $ \hat{y}_{ij}=x_{ij} I (\lvert x_{ij}\rvert \geq C) -\bbE x_{ij}  I (\lvert x_{ij}\rvert \geq C) $  for some constant $ C $ and define  
	$ \hat{\bbX}= (\hat{x}_{ij} )_{p\times n} $, $ \bbY = \bbX - \hat{\bbX} $ and
	$$ \hat{\bbB}_n = \dfrac{1}{n}(\bbR_n + \bbT^{1/2}_n \hat{\bbX})(\bbR_n + \bbT^{1/2}_n \hat{\bbX})^*.  $$
	By Lemma \ref{lemma5.11} and Lemma \ref{lemma2.5}, we have
	\begin{align*}
	\max_{i\leq p} \lvert \lambda_i (\bbB_n) - \lambda_i(\hat{\bbB}_n)\rvert  \leq  \dfrac{1}{\sqrt{n}}\lVert \bbT^{1/2}_n \bbY \rVert  
	\leq \sqrt{\lVert \bbT_n \rVert} (1 + \sqrt{y}) \sqrt{\bbE \lvert \bbX^2 \rvert I( \lvert \bbX \rvert  \geq  C)},
	\end{align*}
	which can be arbitrarily small when $ C $ is large.
	
	Choosing $ \left[ a', b' \right]   $ and $ (c', d') $ such that $ c < c' < a' < a < b < b' < d' < d $. Select $ C $ large enough such that 
	$ \sqrt{\lVert \bbT_n \rVert}  (1+\sqrt{y}) \sqrt{\bbE \lvert X^2 \rvert I (\lvert X \rvert \geq C) }$ is smaller than the smallest gap among 
	$ c < c' < a' < a < b < b' < d' < d $. 
	The intervals $ \left[ a', b' \right]   $ and $ (c', d') $ satisfy the conditions of Theorem 
	$ \ref{th7221} $  for the matrix $ \hat{\bbB}_n $.
	If we proved that no eigenvalues of $ \hat{\bbB}_n  $ are falling into the interval 
	$ [a', b'] $, then there will be no eigenvalues of $ \bbB_n $ falling in $ [a, b] $.
	Therefore, we may prove the Theorem $ \ref{th7221} $ under the additional assumption
	that the random variables are uniformly bounded.
	
	Similar to \cite{bai2012no}, we need to establish an estimate like Theorem 1.2 in  \cite{bai2012no}.  Then we will be devoted to proving the following.
	
	\begin{theorem}\label{th1.2}
		Let $ z=x+iv_n $ with $ v_n = n^{- \delta} $, where $ \delta > 0 $. Then for some small but constant 
		$ \delta >0$, 
		\begin{align}  \label{7443}
		\sup_{x\in\left[ a, b \right]  } n v_n \lvert s_n (x + iv_n) - s^0_n (x + i v_n) \rvert  \stackrel{a.s.}{\longrightarrow} 0, 
		\end{align} 
		where $ s^0_n(z) $ is the solution to $ \eqref{712.1} $ with $ (y, H) $ replaced by 
		$ (y_n, H_n) $.
	\end{theorem}
	
	As shown in \cite{bai1998no,bai2012no}, Theorem $ \ref{th1.2} $ can prove Theorem $ \ref{th7221} $. 
	Let $ [a^{'} , b^{'} ] \subset (c, d) $ for which $ a^{'}< a $, $b^{'} > b $. Then from
	Theorem \ref{th1.2}, it is straightforward to argue (more details will be presented in Section 5)
	\begin{align*}
	&\sup_{x\in[a,b]} \left| \int \dfrac{I([a^{'}, b^{'}]^c)\mathrm{d}(F^{\bbB_n}(\lambda) - F^{y_n,H_n}(\lambda))}{	((x - \lambda)^2 + v_n^2 )((x - \lambda)^2 + 2v_n^2) \cdots ((x - \lambda)^2 + kv_n^2)} \right. \\ & \left.
	+ \sum_{ \lambda_j \in [a',b'] }\dfrac{v_n^{2k}}{ ((x - \lambda)^2 + v_n^2 )((x - \lambda)^2 + 2v_n^2) \cdots ((x - \lambda)^2 + kv_n^2)}     \right|=o(1)	,  \ \mathrm{a.s.} 
	\end{align*}  
	where the $\lambda_j $ are the eigenvalues of $ \bbB_n $.  Since the integral converges a.s. to zero, one
	can argue, by contradiction that there can be no eigenvalues of $ \bbB_n $ in $ [a, b] $
	for all $ n $ large.
	
	We will prove Theorem \ref{th1.2} in Sections 3-5. Before the proof, some lemmas needed in the proof are listed next, and Section 6 contains the mathematical tools in the previous sections.


	\begin{lemma} \label{lemma722}
		Under the conditions of Theorem 2.1 in \cite{zhou2022limiting}, when $ y \leq 1 $, the real part of $ 1 + yg(z) $ is always positive.
	\end{lemma}
	\begin{proof}
		Write $ s(z) = s_1 + is_2 $ and $ g = g_1 + ig_2 $, where $ s_1  $ ,$ s_2 $, $ g_1 $ and $ g_2 $ are both real.
		When $ x \leq 0 $ and $ z = x + iv  $ with $ v \geq 0 $, $ \Re g_n(z) \geq 0 $, as a limit, $ g_1(z) \geq  0 $, hence $ 1 + yg_1 > 0 $. 
		Therefore, we only need to show the lemma for $ \Re z > 0 $.
		When $ x = \Re z \to \infty  $ and $ v =\Im z \geq 0 $ fixed, 
		one can easily show that $ g(z) \to 0 $ and hence $ 1+yg_1(z) \to 1 > 0 $ for all large $ x $. 
		Therefore, there exists a constant $ \mu> 0 $ such that when $ x > \mu $, 
		$ 1 + yg_1(z) > 0 $. 
		Since $ g_1(z) $ is a continuous function of $ x $, if the lemma is untrue, then there exists a $ x > 0 $ and a $ v \leq 0 $ such that $ 1 + yg_1(z) = 0 $. 
		We will show the lemma by deriving a contradiction to this assumption.
		
		Comparing the real parts of two sides of the two equations of $ \eqref{712.1} $, 
		we obtain
		$$  s_1 = -B_0 x - yB_1(xs_1 - vs_2) + B_1(1 - y),   $$
		$$  -\dfrac{1}{y}= -B_1x - yB_2(xs_1 - vs_2) + B_2(1 - y),  $$
		where
		$$  B_j = \int  \frac{t^j \mathrm{d}H(u, t)}{\left|   \dfrac{u}{1+yg(z)}- (1 + ys(z)t)z + t(1 - y)\right|  ^2} , j = 0, 1, 2.  $$
		
		From the two equations, we obtain
		\begin{align*}
		s_2 &= \frac{( \dfrac{1}{y} - xB_1 + B_2(1 -y))(1 + xyB_1) - xyB_2(-xB_0 + B_1(1 - y))}{-vyB_2(1 + xyB_1) + vxy^2 B_1 B_2}\\
		&= \frac{-( \dfrac{1}{y} + B_2(1 - y)) + x^2 y B^2_1 - x^2 y B_2 B_0}{vy B_2 }.
		\end{align*}
		We get a contradictory by Cauchy-Schwarz inequality $ B^2_1 \leq B_0 B_1  $ and $ -( \dfrac{1}{y}+ B_2(1 -y)) < 0$.  
		The proof of the lemma is complete.
	\end{proof}
	
	\begin{lemma} \label{lemma723}
		Under the conditions of Theorem 2.1 in \cite{zhou2022limiting}, for each large $ n$, the density of 
		$ F_0^{y_n , H_n}$, is bounded by $ Kx^{-1/2 }$ for some constant $ K $. 
		Hence, $ F_0^{y_n , H_n}$ satisfies the Lipschitz condition with index 
		$ \frac{1}{2} $.
	\end{lemma}	
	\begin{proof}
		By the inversion formula $ \eqref{1.3} $, the density of $ F_0^{y_n , H_n}$ equals 
		$ \pi ^{-1} \Im s^0_n(x)  $. 
		Therefore, the first assertion follows by estimating the bound of $ s^0_n(x) $. 
		When $ \Im s^0_n(x) = 0 $, the density is $ 0 $, it is surely bounded. 
		We only need to consider the case where its imaginary part is positive. 
		Note that $ s^0_n $ satisfies the first equation in $ \eqref{712.1} $, taking its imaginary part we obtain
		\begin{align*}
		\Im s^0_n(x) =&\int \frac{\frac{u\Im(g^0_n(x))}{ \lvert 1+y_n g^0_n \rvert^2} \mathrm{d} H_n(u, t)}{ \left\lvert  \dfrac{u}{1+y_n g^0_n}- (1 + y_n st)x + t(1 - y_n)\right\rvert ^2}\\
		&+ \int \frac{y_n x t \Im(s^0_n(x)) \mathrm{d} H_n(u, t)}{\left \lvert  \dfrac{u}{1+y_n g^0_n}- (1 + y_n st)x + t(1 - y_n)\right\rvert ^2}.
		\end{align*}
		Dropping the non-negative first term and then eliminating $ \Im(s^0_n(x)) $, we obtain
		\begin{align}\label{7444}
		1 > \int \frac{ y_n x t \mathrm{d} H_n(u, t)}{ \left\lvert  \dfrac{u}{1+yg^0_n}- (1 + y_n st)x + t(1 - y_n)\right\rvert ^2}.
		\end{align}
		
		On the other hand, by the first equation of $ \eqref{712.1} $ and the Cauchy-Schwarz inequality,
		we obtain
		\begin{align*}
		\lvert s^0_n(x) \rvert &\leq \left(  \int \dfrac{1}{t} \mathrm{d} H_n(u, t)  \int \frac{ t \mathrm{d} H_n(u, t) }{\left\lvert  \dfrac{u}{1+yg^0_n}- (1 + y_n st)x + t(1 - y_n)\right\rvert ^2}\right) ^{1/2}\leq \lambda_{-1} \dfrac{1}{\sqrt{y_n x}}.
		\end{align*}
		Hence, the first assertion of the lemma is proved. 
		The second assertion is a natural consequence of the first. 
		In fact, for any $ 0 \leq a < b < \infty $, we have
		\begin{align*}
		F_0^{y_n, H_n} (b) - F_0^{y_n, H_n}  (a) &= \int_{a}^{b} f_0^{y_n, H_n}(t) \mathrm{d} t \leq  K \int_{a}^{b} t^{-1/2} \mathrm{d} t\\
		&\leq 2K(\sqrt{b}-\sqrt{a}) = 2K \dfrac{b -a }{\sqrt{b}+\sqrt{a}} \leq  2K \lvert b - a\rvert ^{1/2}.
		\end{align*}
		The proof is complete.
	\end{proof}
	
	\begin{lemma}\label{lemma724}
		Suppose $ z = x + iv_n  $ and $ v_n \geq  n^{-1/10} $,
		$$  \sup_{x\in \left[-A,A \right]} \lvert s_n(z) - s^0_n(z) \rvert \leq \sqrt{v_n},  $$
		where $ A = v_n^{-2} $. 
		Then  
		\begin{align}\label{7445}
		F_0^{y_n,H_n} (\left[ a, b \right]) \leq \sqrt{v_n},
		\end{align}
		for all large $ n $.
	\end{lemma}
	\begin{proof}
		We shall use Lemma $ \ref{lemma3.1} $ to prove this lemma. 
		Choose $ B = v_n^{-1.5} $.
		Then, with probability one, for all large $ n $, 
		\begin{align}\label{7446}
		F_n(\left[-B, B\right]^c) \leq  B^{-1} \dfrac{1}{p} \mathrm{tr} \bbB_n \leq  KB^{-1}
		\end{align}
		for some absolute constant $ K $. Similarly, we can prove
		\begin{align}\label{7447}
		F_0^{y_n,H_n}(\left[-B, B\right]^c) \leq  KB^{-1}
		\end{align}
		since $ F_n \stackrel{\mathcal{D}}{\longrightarrow} F_0^{y,H} $ and 
		$ F_0^{y_n,H_n} \stackrel{\mathcal{D}}{\longrightarrow}  F_0^{y,H}. $
		
		By the selection of $ A $ and $ B $, the parameter $ \kappa $ satisfies $ \eqref{B210} $. Applying Lemma $ \ref{lemma3.1} $, the lemma follows.
	\end{proof}
	
	\begin{lemma}\label{lemma725}
		Under the conditions of Theorem $ \ref{th7221} $, the spectral norm of
		$ (\bbK - z\bbI)^{-1} $ is bounded for all $ x \in \left[a, b \right] $ and $ z = x + iv_n $, where
		$$  \bbK = \frac{n^{-1}\bbR_n \bbR_n^*}{1 + y_n  \bbE g_n(z) } - z \bbE \underline{s}_n(z)\bbT_n. $$
	\end{lemma}
	\begin{proof}
		Since $ \left[a, b \right] $ is a subinterval of $ (c, d) \subset S_{\underline{F}_n} $
		for all large $ n $, there exists $ 0 < \varepsilon_1 < \varepsilon_2  $ such that $ \left[a'', b'' \right] = [a -\varepsilon_2, b + \varepsilon_2] \subset (c, d)$. 
		Write $ (\underline{s}_0(x), \underline{g}_0(x)) $ be the extended solution, i.e., limits of regular solutions for $z\in\mathbb C^+$, to \eqref{713}. 
		By part (b) of Theorem 4 in \cite{zhou2023analysis},
		for each
		$ x \in  \left[a'', b'' \right]$, we have
		$$ \inf \{ \lvert u\underline{g}_0(x) + t \underline{s}_0(x) + 1 \rvert : (u, t) \in S_H \} > 0 $$
		and consequently
		$$ \inf_{x\in \left[a'', b'' \right] } \{ \lvert u\underline{g}_0(x) + t \underline{s}_0(x) + 1 \rvert : (u, t) \in S_H \} > 0 .$$
		
		Since $ (y_n, H_n) \to (y, H) $, for $ n\to \infty $, there exist $ \underline{s}^0_n $ and $\underline{g}^0_n$ which are extended solutions to the equation \eqref{713} with $ (y, H) $ replaced by $ (y_n, H_n) $
		and
		\begin{align}\label{7448}
		\delta_n := \inf_{x\in \left[a', b' \right] } \min_{i \leq p} \{ \vert u_i \underline{g}^0_n(x) + t_i \underline{s}^0_n(x) + 1 \rvert : (u_i , t_i) \in  S_{H_n} \} > 0,
		\end{align}
		where $ \left[a', b' \right] = \left[a - \varepsilon_1, b - \varepsilon_1 \right] $.
		
		We claim that there is a positive lower bound for $ \delta_n \geq \delta > 0 $ for all large $ n $. 
		Convert $ (\underline{s}^0_n , \underline{g}^0_n ) $ to $ (s^0_n, g^0_n)  $ which satisfies \eqref{712.1}. 
		Since the integrands of the two integrals are uniformly bounded, by DCT, we conclude that $ s^0_n(z) \to  s_0(z) $ and $ g^0_n(z) \to  g_0(z) $ for all $ z \in \mathbb{C}^+ $. Thus, $ F^{y_n,H_n} \to F $ weakly.
		Therefore, $ \{F^{y_n,H_n} \} $ is tight and so is $ \{ \underline{F}^{y_n,H_n} \}  $. Thus, there is a constant $ B $ such that $ \underline{F}^{y_n,H_n}(\left[ B, \infty \right)) < 1/3$. 
		Consequently, for $ x \in \left[a'', b'' \right]  $
		\begin{align}\label{7449}
		(\underline{s}^0_n)'(x) = \int \frac{d\underline{F}^{y_n,H_n}(\lambda)}{(\lambda-x)^2}  \geq \int_{0}^{B}  \frac{dF^{y_n,H_n}(\lambda)}{(\lambda-x)^2}  \geq  \dfrac{2 }{3B^2} \geq  m > 0. 
		\end{align}
		It is proved in Theorem 4 of \cite{zhou2023analysis} that $  \underline{g}' = (1 + {\underline{s}}'c B_2)/({\underline{g}}^{-2} - c A_2)  $, so we have $ \underline{g}'_n(x) \geq m $ for some constant $ m > 0 $.
		
		By Theorem 4 in \cite{zhou2023analysis}, for each $ n = 0 $ or large and every supporting point $ (u_i, t_i) $ of $ H_n $, the function $ u_i \underline{g}_n(x) + t_i \underline{s}_n(x) $ is increasing and continuous.
		Therefore, the infimum of 
		$ \vert u_i \underline{g}_n(x) + t_i \underline{s}_n(x) + 1 \rvert  $ for $ x \in [a'', b'']  $ reaches at $ x = a'' $ or $ b''$, say at $ a'' $ and $ u_i \underline{g}_n(x) + t_i \underline{s}_n(x) + 1 > 0 $. 
		Then, for any $ x \in [a', b']  $, we have
		\begin{align}\label{7450}
		&u_i \underline{g}_n(x) + t_i \underline{s}_n(x) + 1 \geq u_i \underline{g}_n(a') + t_i \underline{s}_n(a') + 1 \\ \nonumber
		\geq & u_i \underline{g}_n(a'') + t_i \underline{s}_n(a'') + 1 + \int_{a'}^{a''} (u_i \underline{g}'_n(x) + t_i \underline{s}'_n(x) )	\\ \nonumber
		\geq & 0 + (a' - a'') m \varDelta = (\varepsilon_2- \varepsilon_1) m \varDelta := \delta > 0. 
		\end{align}
		
		If the infimum reaches are $ b''$, 
		then $ u_i \underline{g}_n(x) + t_i \underline{s}_n(x) + 1 < 0 $. 
		One can similarly prove $ \eqref{7450} $, namely,
		$$  u_i \underline{g}_n(x) + t_i \underline{s}_n(x) + 1 \leq -\delta < 0 .  $$
		The assertion is proved.
		
		For any $ z = x + i v_n $, $ x \in [a', b'] $ ,we have
		$$  \lvert \bbE \underline{s}_n(z) - \underline{s}_0(x) \rvert \leq  \lvert \bbE \underline{s}_n(z) - \underline{s}^0_n(x) \rvert  + \lvert  \underline{s}^0_n(z) - \underline{s}_0(x) \rvert \to 0,  $$ 
		$$  \lvert \bbE \underline{g}_n(z) - \underline{g}_0(x) \rvert \leq  \lvert \bbE \underline{g}_n(z) - \underline{g}^0_n(x) \rvert  + \lvert  \underline{g}^0_n(z) - \underline{g}_0(x) \rvert \to 0. $$
		So, we will have
		$$  \min_{i\leq p} \inf_{x \in [a', b'] }  \vert u_i \underline{g}_n(x) + t_i \underline{s}_n(x) + 1 \rvert  > \delta + o(1).  $$
		Hence, we have
		\begin{align}\label{7451}
		\|(\bbK - z I)^{-1} \| \leq  K. 
		\end{align}
	\end{proof}
	
	\begin{lemma}\label{lemma726}
		Under the conditions of Theorem \ref{th7221}, for any bounded non-random vector 
		$ \bbu $,
		$$\sup_{x\in [a,b] } \lvert \bbu^*(\bbB_n - zI)^{-1} \bbu - \bbu^*(\bbK - zI)^{-1} \bbu \rvert  \to  0, \ a.s. $$
		and hence when $ x\in [a,b] $, $ \bbu^*(\bbB_n - zI)^{-1} \bbu $ are uniformly bounded with probability one.
		
		Especially, the conclusion is true for $ \bbu = \dfrac{1}{\sqrt{n}}\bbr_k, k = 1, 2, \dots , p. $
	\end{lemma}
	\begin{proof}
		The proof of the limit can be done by multiplying $ \bbu^* $ from left and $ \bbu $
		from right to the equation 
		$ (\bbK - z\bbI)^{-1} - (\bbB_n - z\bbI)^{-1} $
		and following similar lines as proving the approximation of the Stieltjes transform. 
		The uniformly boundedness of $ \bbu^*(\bbB_n - zI)^{-1} \bbu $ follows from Lemma \ref{lemma725} and the fact that $ \bbu $ is bounded.
	\end{proof}
	
	\begin{definition} \label{def727}
		Random variables $ x_n $ and $ y_n $ are said to be similar in moment denoted as 
		$ x_n  \stackrel{m}{\simeq } y_n $ if for any integer $ \ell \geq  1 $, $ \bbE 
		\vert x_n  - y_n \rvert ^{2\ell} \leq  K_{\ell} n^{-2\ell} v_n^{-6\ell} $.
		If $ y_n $ is non-random and bounded, we say that $ x_n $ is bounded in moment, and denoted as $ x_n \stackrel{m}{<} \infty $. 
		If $ x_n - y_n $ can be written as a sum of a non-negative constant and a random variable which is similar in moment with 0, then we say $ x_n $ is larger than or equal to $ y_n $ in moment, denoted as $ x_n  \stackrel{m}{\geq} y_n  $.
	\end{definition}
	
	\begin{remark} \label{remark728}
		If a random variable is bounded in moment, then it can be basically treated as bounded when computing the expectation of the product of it with a nonnegative random variable $ w $ which is bounded by $ v_n^{-\nu }$
		for some constant $ \nu $, in fact, for any $ \varepsilon > 0 $, its expectation of the product is less than
		$ (\mu_n + \varepsilon)\bbE w + v_n^{-\nu} \bbP(\lvert \iota_n \rvert > \varepsilon) = (\mu_n + \varepsilon)\bbE w + o(1) $, if $ v_n \geq n^{-1/10 } $ and 
		$ \ell $ is chosen large enough.
	\end{remark}
	
	\begin{lemma}\label{lemma729}
		Under the conditions of Theorem \ref{th7221}, for all $ z $ with $ x \in [a, b] $, the quantities $ 1/\beta_k $, $ 1/\breve{\beta}_k $ are uniformly bounded in moment.
	\end{lemma}
	\begin{proof}
		By Theorem 1 in \cite{zhou2023analysis}, there is a constant $ \delta > 0 $ such that for all $ z \in \mathbb{C}^+ $ with $ x \in [a, b] $, $ \underline{g } $ is bounded from above and hence $ \lvert 1 + yg(z) \rvert > \delta > 0 $.
		Since $ F^{y_n, H_n} \to F^{y, H} $, we have $ g^0_n(z) \to g(z) $ uniformly on $ \mathbb{C}^+ $ and thus when $ n $ is large $ \lvert 1 + y_n g^0_n(z)\rvert > \delta  > 0 $, where $ \delta $ is an absolute constant, may
		take different value at different appearances. 
		Consequently, by Lemma \ref{lemma3.2},
		$ \lvert 1 + y_n g^0_{nk}(z)\rvert > \delta  > 0  $.
		
		At first, we point out that these bounds will be used for proving the $ ``b" $
		bounds and hence we may assume that the $ ``a" $ bounds are true and thus
		\begin{align}\label{7452}
		&\bbE \lvert s_n(z) - s^0_n(z) \rvert ^{2\ell } \leq  K_{\ell} n^{-2\ell} v^{-2\ell}_n , \\ \nonumber
		&\bbE \lvert g_n(z) - g^0_n(z) \rvert ^{2\ell } \leq  K_{\ell} n^{-2\ell} v^{-2\ell}_n,
		\end{align}
		if $ z = x + iv_n $ with $ v_n \geq n^{-1/8} $ and choosing $ \ell \geq 4 $.
		Here, the convergence is true uniformly for all $ x\in [e, f]$. 
		Using the notation given in Definition \ref{def727},
		we may say that $ (1 + yg_nk(z))^{-1} $ is uniformly bounded in moment.
		
		If the angle between the complex numbers $ n^{-1} \bbr^*_k (\bbB_{nk} - zI)^{-1} \bbr_k $ and 
		$ 1 + y_n g_{nk}(z) $ is less than or equal to $ 90^{\circ} $, then
		$$ \lvert \breve{\beta}_k \rvert  = \lvert 1 + y_n g_{nk}(z) + n^{-1} \bbr^*_k (\bbB_{nk} - z\bbI)^{-1} \bbr_k\rvert  \geq  \lvert 1 + y_n g_{nk}(z)\rvert  \geq \delta $$
		and thus $ 1/\breve{\beta}_k(x) $ is bounded from above.
		
		Now, assume that the angle between $ n^{-1} \bbr^*_k (\bbB_{nk} - z\bbI)^{-1} \bbr_k $ and $ 1 + y_n g_{nk}(z)  $ is larger than $ 90^{\circ} $. 
		If $ \lvert n^{-1} \bbr^*_k (\bbB_{nk} - z\bbI)^{-1} \bbr_k\rvert < \delta/2 $, then 
		$ \breve{\beta}_k \geq 1/2\delta  $ and thus
		$ 1/\breve{\beta}_k $ is bounded by $ 2/\delta  $ from above. 
		Now, we assume $ \lvert n^{-1} \bbr^*_k (\bbB_{nk} - z\bbI)^{-1} \bbr_k\rvert \geq \delta/2$. 
		By the formula
		\begin{align}\label{7453}
		&n^{-1}\bbr^*_k (\bbB_n - z\bbI)^{-1} \bbr_k  \\ \nonumber
		= &n^{-1}\bbr^*_k (\bbB_{nk} - z\bbI)^{-1} \bbr_k - \frac{n^{-2} \bbr^*_k (\bbB_{nk} - z\bbI)^{-1} \alpha \alpha^* (\bbB_{nk} - z\bbI)^{-1} \bbr_k}{\beta_k} \\ \nonumber
		\stackrel{m}{=} & n^{-1}\bbr^*_k (\bbB_{nk} - z\bbI)^{-1} \bbr_k - \frac{n^{-2} \bbr^*_k (\bbB_{nk} - z\bbI)^{-1} \bbr_k \bbr_k^* (\bbB_{nk} - z\bbI)^{-1} \bbr_k}{\breve{\beta}_k} \\ \nonumber
		\stackrel{m}{=} & \frac{n^{-1} \bbr^*_k (\bbB_{nk} - z\bbI)^{-1} \bbr_k  (1+n^{-1}\bbx_k^*\bbT^{1/2}_n (\bbB_{nk} - z\bbI)^{-1} \bbT^{1/2}_n\bbx_k)}{\breve{\beta}_k} \\ \nonumber 
		\stackrel{m}{=} &  \frac{n^{-1} \bbr^*_k (\bbB_{nk} - z\bbI)^{-1} \bbr_k  (1 + y_n g^0_n(z))}{\breve{\beta}_k},
		\end{align}
		where we have used Lemma \ref{lemma3.2} for
		\begin{align*}
		n^{-1}\bbx_k^*\bbT^{1/2}_n (\bbB_{nk} - z\bbI)^{-1} \bbT^{1/2}_n\bbx_k 
		& \stackrel{m}{=}  n^{-1} \mathrm{tr} \bbT_n (\bbB_{nk} - z\bbI)^{-1}\\
		& \stackrel{m}{=} y_n g_n(z) \stackrel{m}{=}  y g^0_n(z).
		\end{align*}
		
		From Lemma \ref{lemma726}, $ n^{-1}\bbr^*_k (\bbB_n - z\bbI)^{-1} \bbr_k $ is bounded from above, consequently, $ 1/\breve{\beta}_k $ is bounded from above.
		
		Since $ \beta_k\stackrel{m}{=} \breve{\beta}_k  $, $ 1/\beta_k $ is bounded in moment from above.
	\end{proof}
	Similar to \eqref{7453}, one may establish the following lemma.
	
	\begin{lemma}\label{lemma730}
		For any random vector $ \bbv $ which is of bounded norm and independent of 
		$ \bbr_k $, for any $ x \in [a, b] $ we have
		$$ 	\dfrac{ \bbv^* \bbD^{-1}_{nk} \bbr_k }{\breve{\beta}_k} \stackrel{m}{=}  \dfrac{ \bbv^* \bbD^{-1}_{n} \bbr_k}{1 + y g_0(z)  } $$
	\end{lemma}
	
	\begin{lemma}\label{lemma731}
		Under the conditions of Theorem \ref{th7221}, for any integer $ \ell \geq 1 $
		and non-random vector $ \bbalp= (a_1, \dots , a_p)'  $ 
		$$ \bbE \lvert \bbalp^* \bbx_k \rvert ^{2\ell} \leq K_{\ell} \left[ \left( \sum_{j=1}^{p} \lvert a^2_j \rvert \right) ^{\ell} + \sum_{j=1}^{p} \lvert a^{2\ell}_j \rvert  \right]  $$
	\end{lemma}
	\begin{proof}
		This is a consequence of Burkholder inequality, Lemma \ref{lemma2.13} and the
		facts that $ \bbE \lvert x^2_{kj} \rvert  \leq 1 $ and $\lvert x_{kj} \rvert \leq C $, after the truncation.
	\end{proof}
	
	\begin{lemma}\label{lemma732}
		Under the conditions of Theorem \ref{th7221}, for any integer $ \ell \geq 1 $
		and non-random Hermitian matrix $ \bbM = (m_{ij} ) $ 
		$$  \bbE \left|  \bbx^*_k \bbM \bbx_k - \sum_{j=1}^{p} m_{jj} \bbE \lvert x_{kj} ^2  \rvert \right|^{2\ell} \leq \le  K_{\ell} [(\mathrm{tr} \bbM \bbM^*)^{\ell} + \mathrm{tr}(\bbM \bbM^*)^{\ell} ] . $$
	\end{lemma}
	\begin{proof}
		Write
		$$ \bbx^*_k \bbM \bbx_k- \sum_{j=1}^{p} m_{jj} \bbE \lvert x_{kj} ^2  \rvert = \sum_{j=1}^{p} m_{jj} (\lvert x^2_{kj} \rvert - \bbE \lvert x^2_{kj} \rvert )+2 \sum_{j_1=2}^{p} \sum_{j_2=1}^{j_1-1} \Re(m_{j_1 j_2} \overline{x}_{k j_1} x_{k j_2}). $$
		Applying Burkholder inequality to the first term, we obtain
		$$  \bbE \left(  \sum_{j=1}^{p} m_{jj} (\lvert x^2_{kj} \rvert - \bbE \lvert x^2_{kj} \rvert ) \right)^{2\ell}  
		\leq K_{\ell}\left[ \left( \sum_{j=1}^{p} \lvert m_{jj}\rvert^2 \right)^{\ell}  + \sum_{j=1}^{p} \lvert m_{jj} \rvert^{2\ell}  \right]  
		\leq K_{\ell} (\mathrm{tr} \bbM \bbM^*)^{\ell} .  $$
		Applying Burkholder inequality to the second term, we obtain
		\begin{align*}
		&\bbE \left( 2 \sum_{j_1=2}^{p} \sum_{j_2=1}^{j_1-1} \Re(m_{j_1 j_2} \overline{x}_{k j_1} x_{k j_2}) \right) ^{2\ell}\\
		\leq & K_{\ell}  \left[ \left( \sum_{j_1=2}^{p} \bbE \left|  \sum_{j_2=1}^{j_1-1} m_{j_1 j_2} x_{k j_2}\right|^2 \right)^{\ell}  + \sum_{j_1=2}^{p} \bbE \left|  \sum_{j_2=1}^{j_1-1} m_{j_1 j_2} x_{k j_2}\right|^{2\ell}  \right]  .
		\end{align*}
		By independence and Jensen inequality, the first term above is less than or
		equal to
		\begin{align*}
		& K_{\ell} \left( \sum_{j_1=1}^{p} \bbE \left|  \sum_{j_2=1}^{p} m_{j_1 j_2} x_{k j_2}\right|^2 \right)^{\ell} 
		\leq  K_{\ell} \bbE \left( \sum_{j_1=1}^{p} \left|  \sum_{j_2=1}^{p} m_{j_1 j_2} x_{k j_2}\right|^2 \right)^{\ell} \\
		=&K_{\ell}  \bbE (\bbx_k^*\bbM \bbM^*\bbx_k)^{\ell}
		\leq  K_{\ell}[(\mathrm{tr} \bbM \bbM^*)^{\ell} + \mathrm{tr}(\bbM \bbM^*)^{\ell} ] ,
		\end{align*}
		where the last inequality follows by induction. 
		Again, by Burkholder inequality to the second term, we have
		\begin{align*}
		& \sum_{j_1=1}^{p} \bbE \left|  \sum_{j_2=1}^{p} m_{j_1 j_2} x_{k j_2}\right|^{2\ell }
		\leq K_{\ell} \sum_{j_1=1}^{p}  \left[ \left(\sum_{j_2=1}^{p} \lvert m_{j_1 j_2} \rvert^2 \right)^{\ell} + \sum_{j_2=1}^{p} \lvert m_{j_1 j_2} \rvert^{2\ell }\right]  \\
		\leq & K_{\ell}[(\mathrm{tr} \bbM \bbM^*)^{\ell} + \mathrm{tr}(\bbM \bbM^*)^{\ell} ]. 
		\end{align*}
		Collect the inequalities above, the proof of the lemma is complete.
	\end{proof}
	
	\section{Convergence of the Random Part}
	
	Constants appearing in inequalities are designated by $ K $, sometimes subscripted. They are nonrandom and may differ from one appearance to the next.
	
	We first introduce some notation to simplify the writing. 
	Denote $ \bbD_n = \bbB_n - z\bbI $, $ \bbD_{nk} = \bbB_{nk} - z\bbI = \bbD_n - \alpha_k \alpha_k^* $, $ \bbQ = \bbK - z\bbI $, where
	$$  \bbB_{nk} = \bbB_n - \alpha_k \alpha_k^* , ~~\alpha_k = \dfrac{1}{\sqrt{n}} (\bbr_k + \bbT^{1/2} \bbx_k),$$
	$$ \bbK= \dfrac{\frac{1}{n}\bbR_n \bbR_n^*}{1+y \bbE g_n(z)} - z\bbE \underline{s}_n(z) T_n~~\mbox{and}~~ \underline{s}_n= -\dfrac{1-y_n}{z}+y_n s_n(z)  .$$ 
	
	Also, let $ \bbE_0(\cdot) $ be the expectation and $ \bbE_k(\cdot) $ be the conditional expectation with respect to the $\sigma $-field generated by the random variables $ \{ x_{ij}, i,j>k\} $. We employ the martingale technique to decompose the random part
	$ s_n - \bbE s_n $ as a sum of martingale differences
	\begin{align} \label{7454}
	s_n - \bbE s_n &= \sum_{k=1}^{n} (\bbE_k - \bbE_{k-1} ) s_n =  \sum_{k=1}^{n}
	(\bbE_k - \bbE_{k-1})(s_n - s_{nk})\\ \nonumber
	& = \dfrac{1}{p} \sum_{k=1}^{n} (\bbE_k - \bbE_{k-1})\alpha_k^* \bbD^{-2}_{nk} \alpha_k \beta _k^{-1} \\ \nonumber
	&\stackrel{m}= \dfrac{1}{p} \sum_{k=1}^{n} (\bbE_k - \bbE_{k-1})(\sigma_{nk} \breve{\beta}_k^{-1} - \overline{\sigma}_{nk}  \breve{\beta}_k^{-1} \beta_k^{-1} \varepsilon_k), 
	\end{align}
	where
	\begin{align*}
	\beta_k &= 1 + n^{-1} \alpha_k^*\bbD_{nk}^{-1} \alpha_k ,\\
	\breve{\beta}_k &= 1 + n^{-1}\bbr^*_k \bbT^{1/2} \bbD^{-1}_{nk} \bbr_k + n^{-1} \mathrm{tr} \bbT \bbD^{-1}_{nk} ,\\
	\varepsilon_k &= \beta_k^{-1} - \breve{\beta}_k^{-1} \\
	\overline{\sigma}_{nk} &= n^{-1}\bbr^*_ k \bbD^{-2}_{nk} \bbr_k + n^{-1} \mathrm{tr}\bbT \bbD^{-2}_{nk},\\
	\sigma_{nk} &= n^{-1}(\alpha_k^* \bbD^{-2}_{nk} \alpha_k) - \overline{\sigma}_{nk}.
	\end{align*} 
	
	In the following, we will show for $ v_n = \kappa n^{-1/m} $, $ m $ suitable large, and any $ \ell \geq 1 $
	\begin{align} \label{7455}
	\bbE \lvert s_n(z) - \bbE s_n(z)\rvert^{2\ell}  \leq
	\left\{
	\begin{array}{l}
	K_{\ell} v_n^{-6\ell} n^{-2\ell}  \ (a)\\
	K_{\ell}  n^{-2\ell} ,    \ \  \quad \ (b)\\
	\end{array} \right.
	\end{align} 
	and by a similar procedure
	\begin{align} \label{7456}
	\bbE \lvert g_n(z) - \bbE g_n(z)\rvert^{2\ell} \leq
	\left\{
	\begin{array}{l}
	K_{\ell} v_n^{-6\ell} n^{-2\ell}  \ (a)\\
	K_{\ell}  n^{-2\ell} ,    \ \  \quad \ (b)\\
	\end{array} \right.
	\end{align} 
	where $ g_n(z) = p^{-1} \mathrm{tr}\bbT \bbD^{-1}_n $. And for any bounded nonrandom vector $ \bbu  $
	\begin{align}\label{7457}
	\bbE \lvert \bbu^*\bbD_n^{-1} \bbu - \bbE \bbu^* \bbD_n^{-1} \bbu \rvert ^{2\ell} \leq  K_{\ell} n^{-\ell} v^{-4\ell}. 
	\end{align}
	
	The $ ``a" $ bound holds uniformly for all $ x \in [e, f] $, it contributes to a preliminary estimate for the convergence of $ s_n(x + iv_n) - s^0_n(x + iv_n) $. 
	Based on the preliminary estimation, we will establish a convergence rate of the ESD 
	$ \| F_n - F_0^{y_n,H_n }\| = O(\sqrt{v_n}) $. 
	The $ ``b" $ bound holds uniformly for all $ x \in [a, b] $, it is considered as a refinement of the estimate of the convergence of $ s_n(x + i\widetilde{v}_n) - \bbE s_n(x + i\widetilde{v}_n) $. 
	The refined convergence rate will be proved under an additional condition that $ F_0^{y_n,H_n }([a', b']) = O(\widetilde{v}_n^4) $ which is established based on the $ ``a" $ bound, where $ \widetilde{v}_n = \sqrt[8]{v_n}  $ and $ [a', b'] = [a -\varepsilon, b + \varepsilon] \subset  (c, d) \subset  S_{F_0^{y_n,H_n }} , \varepsilon > 0 $. 
	Notice these bounds hold for all $ \ell \geq 1 $ once they are shown to be true for sufficiently large $ \ell $. 
	For brevity, we use an abused notation that simplifies $ \widetilde{v}_n $ as $ v_n $. 
	That means, the $ v_n $'s in $ ``b" $ bounds stands for $ \sqrt[8]{v_n}  $ as those in 
	$ ``a" $  bounds. 
	Namely in the proofs of $ ``b" $ bounds will be proceeded under the additional condition that $ F_0^{y_n,H_n }([a, b]) = O(v_n^4) $.
	
	Before proceeding, we introduce some lemmas for the next proofs.
	We first establish $ \eqref{7455}(a) $. 
	Splitting the interval $ [e, f] $ into $ n $ equal parts and write the set of splitting points as $ S_n $. So,
	\begin{align*}
	& \left|  \sup_{x \in [e,f]} \lvert s_n(z) - \bbE s_n(z)\rvert - \max_{x\in S_n} \lvert s_n(z) - \bbE s_n(z)\rvert  \right|  \\
	\leq & \max_{x\in S_n} \sup_{\lvert x_1 -x_2 \rvert \leq (f- e)/2n}
	[ \lvert s_n(z_1) - s_n(z_2)\rvert  + \lvert \bbE s_n(z_1) - \bbE s_n(z_2) \rvert ] \\
	\leq& (f - e)n^{-1} v_n^{-2}.
	\end{align*}
	So, we have
	\begin{align*}
	&\sup_{x \in [e,f]} \bbE \lvert \underline{s}_n(z) - \bbE \underline{s}_n(z) \rvert ^{2\ell} \\
	\leq &  \max_{x\in S_n}  \bbE \lvert \underline{s}_n(z) - \bbE \underline{s}_n(z) \rvert^{2\ell}  + K_{\ell}n^{-2\ell} v_n^{-4\ell}.
	\end{align*}
	
	Note that $ \Im(\breve{\beta}_k) = v_n(n^{-1} \lVert (\bbD^{-1}_{nk} \bbr_k \rVert^2 +n^{-1} \mathrm{tr} [\bbD^{-1}_{nk}  \overline{\bbD}^{-1}_{nk}]) \geq v_n \lvert \overline{\sigma}_{nk} \rvert $ and hence $ \lvert \overline{\sigma}_{nk}\breve{\beta}_k^{-1} \rvert \leq  1/v_n $, where $ \overline{\bbD} $ denotes the conjugate transpose of $ \bbD $.
	Furthermore, it is already known that $ \lvert \overline{\sigma}_{nk} \rvert \leq Kv_n^{-2} $ and 
	$ \lvert \beta_k^{-1} \rvert  \leq  \lvert z \rvert v_n^{-1} $. 
	Then, by Burkholder inequality, i.e. Lemma \ref{lemma2.13}, and the
	decomposition \eqref{7454}, for $ x \in S_n $, we have for $ v_n = n^{-\delta} $ with 
	$ \delta < 1/(6\ell+ 1) $,
	\begin{align}\label{7458}
	&\bbE \lvert s_n(z) - \bbE s_n(z)\rvert^{2\ell} \\ \nonumber
	\leq& K_{\ell} p^{-2\ell} \left(  \sum_{k=1}^{n} \bbE \lvert \sigma_{nk} \breve{\beta}_k^{-1} - \overline{\sigma}_{nk} \breve{\beta}_k^{-1} \beta_k^{-1} \varepsilon_k \rvert^2 \right)^{\ell}  \\ \nonumber
	& +K_{\ell} p^{-2\ell}  \sum_{k=1}^{n} \bbE \lvert \sigma_{nk} \breve{\beta}_k^{-1} - \overline{\sigma}_{nk} \breve{\beta}_k^{-1} \beta_k^{-1} \varepsilon_k \rvert^{2\ell}\\ \nonumber
	\leq& K_{\ell} p^{-2\ell} \left(  \sum_{k=1}^{n} v_n^{-2} \bbE \lvert \sigma_{nk} \rvert^2+ v_n^{-4} \bbE \lvert \varepsilon_k \rvert^2   \right)^{\ell}  \\ \nonumber   
	& +K_{\ell} p^{-2\ell} \sum_{k=1}^{n} [v_n^{2\ell} \bbE \lvert \sigma_{nk} \rvert^{2\ell} + v_n^{-4\ell}\bbE \lvert \varepsilon_k \rvert ^{2\ell}].
	\end{align}
	Furthermore, we have
	\begin{align}\label{7459}
	\bbE \lvert \sigma_{nk}^2 \rvert &\leq 2n^{-2} [4\bbE \lvert \bbr^*_k \bbD^{-2}_{nk} \bbT^{1/2} \bbx_k \rvert^2 +\bbE \lvert \bbx^*_k \bbT^{1/2} \bbD^{-2}_{nk}\bbT^{1/2} \bbx_k - \mathrm{tr} \bbT \bbD^{-1}_{nk} \rvert^2] \\ \nonumber
	& \leq  2n^{-2} [K\bbE \bbr^*_k \bbD^{-2}_{nk} \bbT_n \overline{\bbD}^{-2}_{nk} \bbr_k + \bbE \mathrm{tr}\bbT \bbD^{-2}_{nk} \bbT_n \overline{\bbD}^{-2}_{nk}] \leq Kn^{-1} v^{-4}_n,
	\end{align}
	and
	\begin{align} \label{7460}
	\bbE \lvert \varepsilon_k^2 \rvert & \leq  2n^{-2} [K \bbE \lvert \bbr^*_k \bbD^{-1}_{nk} \bbT^{1/2} \bbx_k \rvert^2 + \bbE \lvert \bbx^*_k \bbT^{1/2} \bbD^{-1}_{nk} \bbT^{1/2} \bbx_k -  \mathrm{tr}\bbT \bbD^{-1}_{nk} \rvert^2] \\
	\nonumber
	& \leq 2n^{-2} [K\bbE \bbr^*_k \bbD^{-1}_{nk} \bbT_n \overline{\bbD}^{-1}_{nk} \bbr_k + \bbE \mathrm{tr}\bbT \bbD^{-1}_{nk} \bbT_n \overline{\bbD}^{-1}_{nk}] \leq Kn^{-1} v^{-2}_n,
	\end{align}
	and by Lemmas \ref{lemmab26}, \ref{lemma731} and \ref{lemma732}.
	\begin{align}\label{7461}
	\bbE \lvert \sigma_{nk}^{2\ell} \rvert & \leq 2^{\ell} n^{-2\ell}
	[2\bbE \lvert \bbr^*_k \bbD^{-2}_{nk} \bbT^{1/2} \bbx_k \rvert^{2\ell}
	+\bbE \lvert \bbx^*_k \bbT^{1/2} \bbD^{-2}_{nk}\bbT^{1/2} \bbx_k - \mathrm{tr} \bbT \bbD^{-2}_{nk} \rvert^{2\ell}]  \\ \nonumber
	& \leq 2^{2\ell} n^{-2\ell} [2^{2\ell} \bbE \lVert \bbr^*_k \bbD^{-2}_{nk} \bbT^{1/2}_{n} \rVert ^{2\ell} +\bbE \mathrm{tr}(\bbT^2 \bbD^{-2}_{nk} \bbT_n \overline{\bbD}^{-2}_{nk} )^{\ell}]  \\ \nonumber
	& \leq  o(n^{-\ell} v_n^{-4\ell}) = o(n^{-1} v^{2\ell}_n )
	\end{align}
	and
	\begin{align}\label{7462}
	\bbE \lvert \varepsilon_k^{2\ell} \rvert \leq  K_{\ell} n^{-2\ell+1}  v_n^{-\ell} = o(n^{-1}  v_n^{4\ell} ). 
	\end{align}
	Substituting these into $ \eqref{7458} $, we obtain $ \eqref{7455} (a) $. The proof of $ \eqref{7455} (a) $ is done.
	
	The proofs of $ \eqref{7456} (a) $  and $ \eqref{7457} $ are similar to that $ \eqref{7455} (a) $ and hence omitted.

	Now, let us consider the refinement $ \eqref{7455} (b) $ under the additional condition that $ \underline{F}([a', b']) = o(v^4_n) $. 
	Review the proof of $ \eqref{7455} (a) $ we only need to refine the estimates $ \eqref{7459} $ and $ \eqref{7460} $.
	
	Write the spectral decomposition of $ \bbB_{nk} = \sum_{j=1}^{n}  \lambda_{nkj} \bbv_{kj} \bbv^*_{kj} $ , then we have
	\begin{align*}
	&\dfrac{1}{n} \bbE \mathrm{tr} \bbT \bbD^{-1}_{nk} \bbT_n \overline{\bbD}^{-1}_{nk}
	= \dfrac{1}{n} \sum_{j=1}^{n} (\bbv^*_{kj} \bbT_n \bbv_{kj} )^2 \lvert \lambda_{nkj} - z \rvert ^{-2} \\
	= & \dfrac{1}{n} \sum_{\lambda_{nkj} \notin [a',b']}  (\bbv^*_{kj} \bbT_n \bbv_{kj})^2 \lvert \lambda_{nkj} - z\rvert^{-2} + \dfrac{1}{n} \sum_{\lambda_{nkj} \in [a',b']}  (\bbv^*_{kj} \bbT_n \bbv_{kj})^2 \lvert \lambda_{nkj} - z\rvert^{-2} \\
	\leq & \dfrac{1}{n} \nu^{-2} \sum_{\lambda_{nkj} \notin [a',b']} (\bbv^*_{kj} \bbT_n \bbv_{kj} )^2 + K v_n^{-2}(\underline{F}_n([a', b'] )+ O(p^{-1})) \\
	\le& \frac{K}{\nu^2n}\sum_{\lambda_{nkj} \notin [a',b']}+ O(p^{-1})=K\nu^{-2} \bbE F_n([a',b'])+O(p^{-1})\\
	\le& K\nu^{-2}v_n^4+o(1) ,
	\end{align*}
	where $ \nu = \min(a - a', b' - b) $ and the last step follows from the facts that
	$ \bbv^*_{kj} \bbT_n \bbv_{kj}  $ is bounded since $ \bbT_n $ is bounded in spectral norm and that the difference of the numbers of eigenvalues of $ \bbB_{nk} $ falling in the interval $ [a', b'] $ from that of $ \bbB_n $ is at most $ 1 $ by the Lemma \ref{p528inth}.
	Consequently, we obtain
	\begin{align}\label{7463}
	\dfrac{1}{n} \bbE \mathrm{tr} \bbT_n \bbD^{-1}_{nk} \bbT_n \bbD^{-1}_{nk} = o(1). 
	\end{align}
	Similarly, we can prove
	\begin{align}\label{7464}
	\dfrac{1}{n} \bbE \mathrm{tr} \bbT_n \bbD^{-2}_{nk} \bbT_n \overline{\bbD}^{-2}_{nk} = o(1). 
	\end{align}
	Next, we shall estimate $ n^{-1} \bbr^*_k \bbD^{-1}_{nk} \bbT_n \overline{\bbD}^{-1}_{nk} \bbr_k \leq Kn^{-1} \bbr^*_k \bbD^{-1}_{nk} \overline{\bbD}^{-1}_{nk} \bbr_k $. 
	We use the size reducing formula $ \bbD_n^{-1}  = \bbD_{nk}^{-1} -  \dfrac{\bbD_{nk}^{-1} \alpha_k \alpha_k^*  \bbD_{nk}^{-1} }{\beta_k} $
	and by similar steps
	to $ \eqref{7453} $, we get
	\begin{align*}
	&n^{-1}\bbr^*_k \bbD^{-1}_n \overline{\bbD}^{-1}_n \bbr_k \\
	=& n^{-1}\bbr^*_k \bbD^{-1}_{nk} \overline{\bbD}^{-1}_{nk} \bbr_k 
	- \dfrac{n^{-1}\bbr^*_k \bbD^{-1}_{nk} \alpha_k \alpha_k^* \bbD^{-1}_{nk} \overline{\bbD}^{-1}_{nk} \bbr_k }{\beta_k} \\
	&-\dfrac{n^{-1}\bbr^*_k \bbD^{-1}_{nk} \overline{\bbD}^{-1}_{nk}  \alpha_k \alpha_k^*  \overline{\bbD}^{-1}_{nk} \bbr_k  }{\overline{\beta}_k}  
	+\dfrac{n^{-1}\bbr^*_k \bbD^{-1}_{nk} \alpha_k \alpha_k^* \bbD^{-1}_{nk} \overline{\bbD}^{-1}_{nk}  \alpha_k \alpha_k^*  \overline{\bbD}^{-1}_{nk} \bbr_k  }{\overline{\beta}_k \beta_k} \\
	\stackrel{m}{=} & n^{-1}\bbr^*_k \bbD^{-1}_{nk} \overline{\bbD}^{-1}_{nk} \bbr_k 
	- \dfrac{n^{-1}\bbr^*_k \bbD^{-1}_{nk} \bbr_k \bbr_k^* \bbD^{-1}_{nk} \overline{\bbD}^{-1}_{nk} \bbr_k }{\beta_k} \\
	&-\dfrac{n^{-1}\bbr^*_k \bbD^{-1}_{nk} \overline{\bbD}^{-1}_{nk}  \bbr_k \bbr_k^*  \overline{\bbD}^{-1}_{nk} \bbr_k  }{\overline{\beta}_k}  
	+\dfrac{n^{-1}\bbr^*_k \bbD^{-1}_{nk} \bbr_k \bbr_k^* \bbD^{-1}_{nk} \overline{\bbD}^{-1}_{nk}  \bbr_k\bbr_k^*  \overline{\bbD}^{-1}_{nk} \bbr_k  }{\overline{\beta}_k \beta_k}  \\
	\stackrel{m}{=} & n^{-1}\bbr^*_k \bbD^{-1}_{nk} \overline{\bbD}^{-1}_{nk} \bbr_k \times \left(  1 - \frac{n^{-1}\bbr^*_k \bbD^{-1}_{nk} \bbr_k}{ 1+ n^{-1} \alpha_k^* \bbD^{-1}_{nk} \alpha_k + y_n \bbE g_n }  \right) \\
	&\times \left(  1 - \frac{n^{-1}\bbr^*_k \overline{\bbD}^{-1}_{nk} \bbr_k}{ 1+ n^{-1} \alpha_k^* \overline{\bbD}^{-1}_{nk} \alpha_k + \bbE \overline{g}_n }  \right) \\
	= &\frac{n^{-1}\bbr^*_k \bbD^{-1}_{nk} \overline{\bbD}^{-1}_{nk} \bbr_k \lvert 1 + y_n \bbE g_n \rvert^2 }{\lvert \breve{\beta}_k \rvert^2 } .
	\end{align*}
	The last step follows by Lemma $ \ref{lemma729} $. Therefore, we obtain
	\begin{align}\label{7465}
	\frac{n^{-1}\bbr^*_k \bbD^{-1}_{nk} \overline{\bbD}^{-1}_{nk} \bbr_k}{
		\lvert \breve{\beta}_k \rvert^2 }  
	\stackrel{m}{=}  \frac{n^{-1}\bbr^*_k \bbD^{-1}_{n} \overline{\bbD}^{-1}_{n} \bbr_k}{ \lvert 1 + y_n \bbE g_n \rvert^2 } . 
	\end{align}
	Note that
	$$  \lvert \beta_k \rvert^2 \stackrel{m}{=} \lvert \breve{\beta}_k \rvert^2, \  \text{and} \ 
	1 + y_n \bbE g_n \to 1 + y g_0.  $$
	Hence, by the spectral decomposition of $ \bbB_n = \sum_{j=1}^{n} \lambda_{nj} \bbv_j \bbv_j^* $, we have
	\begin{align}\label{7466}
	&\sum_{k=1}^{n} n^{-2}
	\frac{ \bbr^*_k \bbD^{-1}_{nk} \overline{\bbD}^{-1}_{nk} \bbr_k }{\lvert \breve{\beta}_k \rvert^2 } 
	\leq K \sum_{k=1}^{n} n^{-2} \frac{ \bbr^*_k \bbD^{-1}_{n} \overline{\bbD}^{-1}_{n} \bbr_k}{ \lvert 1 + yg_0 \rvert^2 } 
	\leq K n^{-2} \mathrm{tr} \bbR_n \bbR^*_n \bbD^{-1}_n \overline{\bbD}^{-1}_{n}\\
	\nonumber
	= & K n^{-1} \sum_{j=1}^{p} \bbv_j^* (\bbR_n \bbR^*_n/n) \bbv_j  \lvert \lambda_j - z \rvert^{-2} \\ \nonumber 
	= & K n^{-1} \sum_{ \lambda_j \notin [a',b'] } \bbv_j^* (\bbR_n \bbR^*_n/n) \bbv_j  \lvert \lambda_j - z \rvert^{-2} + K n^{-1} \sum_{ \lambda_j \in [a',b'] } \bbv_j^* (\bbR_n \bbR^*_n/n) \bbv_j  \lvert \lambda_j - z \rvert^{-2} \\ \nonumber
	\leq & K n^{-1} \sum_{ \lambda_j \notin [a',b'] } \bbv_j^* (\bbR_n \bbR^*_n/n) 
	+ K v_n^{-2} F_n([a', b']) \\ \nonumber
	\leq & o_{a.s}(1),
	\end{align}
	where the second inequality follows by $ \lvert 1 + y g_0\rvert  \geq \delta > 0 $ due to Theorem 1 in \cite{zhou2023analysis},  the last step follows since 
	$ \bbv_j^* (\bbR_n \bbR^*_n/n) \bbv_j $ is bounded and $ F_n([a', b'])=o_{a.s}(v^4_n) $.
	
	Finally, similarly applying the size reducing formula to the first $ \bbD^{-1} $ and
	the last $ \overline{\bbD}^{-1} $, we obtain
	\begin{align*}
	&	n^{-1} \bbr^*_k \bbD^{-2}_n \overline{\bbD}^{-2}_n \bbr_k \\
	\stackrel{m}{=} & n^{-1} \bbr^*_k \bbD^{-1}_{nk} \bbD_n^{-1}  \overline{\bbD}^{-1}_{n} \overline{\bbD}^{-1}_{nk} \bbr_k  
	- \dfrac{ n^{-1}  \bbr^*_k \bbD^{-1}_{nk} \bbr_k}{\beta_k} n^{-1} \bbr^*_k \bbD^{-1}_{nk} \bbD_n^{-1}  \overline{\bbD}^{-1}_{n} \overline{\bbD}^{-1}_{nk} \bbr_k \\
	& - n^{-1} \bbr^*_k \bbD^{-1}_{nk} \bbD_n^{-1}  \overline{\bbD}^{-1}_{n} \overline{\bbD}^{-1}_{nk} \bbr_k \dfrac{ n^{-1}  \bbr^*_k \overline{\bbD}^{-1}_{nk} \bbr_k}{\overline{\beta}_k}  
	+ n^{-1} \bbr^*_k \bbD^{-1}_{nk} \bbD_n^{-1}  \overline{\bbD}^{-1}_{n} \overline{\bbD}^{-1}_{nk} \bbr_k  \dfrac{\lvert n^{-1} \bbr^*_k \bbD^{-1}_{nk} \bbr_k \rvert^2 }{\lvert \beta_k \rvert ^2} \\
	\stackrel{m}{=} & n^{-1} \bbr^*_k \bbD^{-1}_{nk} \bbD_n^{-1}  \overline{\bbD}^{-1}_{n} \overline{\bbD}^{-1}_{nk} \bbr_k 
	\left( 1- \dfrac{n^{-1} \bbr^*_k \bbD^{-1}_{nk} \bbr_k}{\beta_k} \right) 
	\left( 1- \dfrac{n^{-1} \bbr^*_k \overline{\bbD}^{-1}_{nk} \bbr_k}{\overline{\beta}_k}   \right) \\
	\stackrel{m}{=} &  n^{-1} \bbr^*_k \bbD^{-1}_{nk} \bbD_n^{-1}  \overline{\bbD}^{-1}_{n} \overline{\bbD}^{-1}_{nk} \bbr_k  \dfrac{\lvert 1 + y \bbE g_n\rvert^2 }{\lvert \breve{\beta}_k \rvert^2}.
	\end{align*}
	Furthermore,
	\begin{align*}
	&\bbr^*_k \bbD^{-1}_{nk} \bbD_n^{-1}  \overline{\bbD}^{-1}_{n} \overline{\bbD}^{-1}_{nk} \bbr_k 
	\\
	\stackrel{m}{=} & \bbr^*_k \bbD^{-2}_{nk} \overline{\bbD}_{nk}^{-2} \bbr_k 
	-2 \Re \left(  \frac{\bbr^*_k \bbD^{-2}_{nk} \overline{\bbD}^{-1}_{nk} \bbr_k}{\beta_k}  \bbr^*_k \bbD^{-1}_{nk} \overline{\bbD}^{-1}_{nk} \bbr_k \right) 
	+ \dfrac{1}{\lvert \beta_k \rvert^2} (\bbr^*_k \bbD^{-1}_{nk} \overline{\bbD}^{-1}_{nk} \bbr_k )^3.
	\end{align*}
	And similarly
	\begin{align*}
	\bbr^*_k \bbD_n^{-2}  \overline{\bbD}^{-1}_{n}  \bbr_k 
	& \stackrel{m}{=}  \frac{\bbr^*_k \bbD^{-1}_{nk} \bbD_n^{-1}\overline{\bbD}^{-1}_{nk} \bbr_k \lvert 1 + y \bbE g_n\rvert^2}{\lvert \beta_k \rvert^2} \\ 
	& \stackrel{m}{=} \frac{\bbr^*_k \bbD^{-2}_{nk} \overline{\bbD}^{-1}_{nk} \bbr_k \lvert 1 + y \bbE g_n\rvert^2}{\lvert \beta_k \rvert^2} 
	- \frac{2 \Re( \bbr^*_k \bbD^{-2}_{nk} \bbr_k) \bbr^*_k \bbD^{-1}_{nk} \overline{\bbD}^{-1}_{n} \bbr_k \lvert 1 + y \bbE g_n\rvert^2}{ \beta \lvert \beta_k \rvert^2}  .
	\end{align*}
	Consequently, we obtain
	\begin{align*}
	&\frac{n^{-1} \bbr^*_k \bbD_{nk}^{-2}  \overline{\bbD}^{-2}_{nk}  \bbr_k }{\lvert \breve{\beta}_k \rvert^2} \\
	\stackrel{m}{=} & \frac{n^{-1} \bbr^*_k \bbD_{nk}^{-1} \bbD_{n}^{-1} \overline{\bbD}^{-1}_{n}  \overline{\bbD}^{-1}_{nk}  \bbr_k }{\lvert \breve{\beta}_k \rvert^2} 
	+ 2 \Re \left(  \frac{\bbr^*_k \bbD^{-2}_{nk} \overline{\bbD}^{-1}_{nk} \bbr_k}{\beta_k \lvert \beta_k \rvert^2}  \bbr^*_k \bbD^{-1}_{nk} \overline{\bbD}^{-1}_{nk} \bbr_k \right) 
	- \dfrac{1}{\lvert \beta_k \rvert^4} (\bbr^*_k \bbD^{-1}_{nk}  \overline{\bbD}^{-1}_{nk} \bbr_k )^3 \\
	\stackrel{m}{=} & \frac{n^{-1} \bbr^*_k \bbD_{n}^{-2}  \overline{\bbD}^{-2}_{n}  \bbr_k }{\lvert 1 + y \bbE g_n \rvert^2} 
	+ 2 \Re \left(  \frac{\bbr^*_k \bbD^{-2}_{n} \overline{\bbD}^{-1}_{n} \bbr_k}{\beta_k \lvert 1 + y \bbE g_n \rvert^2}  \bbr^*_k \bbD^{-1}_{nk} \overline{\bbD}^{-1}_{nk} \bbr_k \right) \\
	&+ 4\Re \left(  \frac{\bbr^*_k \bbD^{-2}_{n} \bbr_k}{\breve{\beta}_k^2
		\lvert \breve{\beta}_k \rvert^2} ( \bbr^*_k \bbD^{-1}_{nk} \overline{\bbD}^{-1}_{nk} \bbr_k )^2 \right)  
	- \dfrac{1}{\lvert \beta_k \rvert^4} (\bbr^*_k \bbD^{-1}_{nk}  \overline{\bbD}^{-1}_{nk} \bbr_k )^3
	\end{align*}
	Therefore, we may similarly prove that
	\begin{align*}
	\bbE \sum_{k=1}^{n} \frac{n^{-2} \bbr^*_k \bbD_{nk}^{-2}  \overline{\bbD}^{-2}_{nk}  \bbr_k }{\lvert \breve{\beta}_k \rvert^2} & \stackrel{m}{\leq}  \bbE \sum_{k=1}^{n}
	\frac{n^{-2} \bbr^*_k \bbD_{n}^{-2}  \overline{\bbD}^{-2}_{n}  \bbr_k }{\lvert 1 + y \bbE g_n\rvert^2} \leq Kn^{-2} \mathrm{tr} (\bbR_n \bbR_n^* \bbD^{-2} \overline{\bbD}^{-2}) \leq o(1).
	\end{align*}
	Substituting these two estimate to $ \eqref{7458} - \eqref{7460}$, the proof of 
	$ \eqref{7455} (b) $ will be complete. 
	The assertions $ \eqref{7456} (a) $  and $ \eqref{7456} (b) $  can be similarly
	proved.

	\section{Convergence of the Nonrandom Part}
	
	Our next goal is to establish the convergence rate of the nonrandom part, i.e., we shall
	prove that
	\begin{align}\label{7467}
	\sup_{x\in[a,b]} \lvert \bbE s_n(x + i v_n) - s^0_n(x + i v_n)\rvert = O(n^{-1}). 
	\end{align}
	This result is not only a necessary step for the proof of $ \eqref{7443} $, it is also
	helpful to the proof of the random part.
	
	In this section, 
	We first prove that
	\begin{align}\label{7468}
	& \bbE s_n(z) - \dfrac{1}{p} \mathrm{tr} \bbQ^{-1} = \omega_{n1}(z) = O(n^{-1}), \\ \nonumber
	& \bbE g_n(z) - \dfrac{1}{p} \mathrm{tr} \bbT_n \bbQ^{-1} = \omega_{n2}(z) = O(n^{-1}), 
	\end{align}
	where $ \bbQ = \bbK - z\bbI $ and
	$$ \bbK = \frac{n^{-1}\bbR_n \bbR_n^*}{1 + y_n \bbE g_n(z)} - z\bbE \underline{s}_n(z)\bbT_n.$$
	In the process of obtaining the LSD of $ \bbB_n $, we proved that
	\begin{align}\label{7214}
	\frac{1}{p}  \mathrm{tr} \bbT^l_n \bbQ^{-1} - \frac{1}{p}\bbE \mathrm{tr} \bbT^l_n(\bbB_n-z\bbI)^{-1} \to 0, \ \mathrm{for} \ \ell=0,1.
	\end{align}
	We only need to refine the order of \eqref{7214}. For the refinement, one needs only examine step by step for all error terms to have the order $ O(n^{-1}) $ when $ x \in [a, b] $.

	Recall the proof of \eqref{7214}, we have
	\begin{align} \label{726}
	&\frac{1}{p} \mathrm{tr} \bbT^\ell_n \bbQ^{-1} - \frac{1}{p} \bbE \mathrm{tr} \bbT^\ell_n(\bbB_n - z\bbI)^{-1}\\ \nonumber
	=& \frac{1}{p}\sum_{k=1}^{n}\bbE \dfrac{\alpha_k^*\bbD_{nk}^{-1}\bbT^\ell_n \bbQ^{-1}\alpha_k }{\beta_k} 
	-\frac{1}{p} \bbE \mathrm{tr} \bbT^\ell_n (\bbB_n - z\bbI)^{-1} \bbK \bbQ^{-1} \\ \nonumber
	=& \frac{1}{p}\sum_{k=1}^{n}\bbE \dfrac{\alpha_k^*\bbD_{nk}^{-1}\bbT^\ell_n\bbQ^{-1} \alpha_k}{\breve{\beta}_k } -\frac{1}{p} \bbE \mathrm{tr} \bbT^\ell_n (\bbB_n - z\bbI)^{-1} \bbK \bbQ^{-1} + o(1)\\ \nonumber
	=& \frac{1}{p} \sum_{k=1}^{n} \bbE  \dfrac{\bbr_k^* \bbQ^{-1} \bbT^\ell_n \bbD_{nk}^{-1} \bbr_k  + \mathrm{tr} [\bbT_n \bbQ^{-1} \bbT^\ell_n \bbD_{nk}^{-1}]}{\breve{\beta}_k} -\frac{1}{p} \bbE \mathrm{tr} \bbT^\ell_n (\bbB_n - z\bbI)^{-1} \bbK \bbQ^{-1} + o(1) \nonumber
	\end{align}

	Using
	$$ (\bbB_n- z\bbI)^{-1} = (\bbB_{nk} - z\bbI)^{-1}-\dfrac{\frac{1}{n}(\bbB_{nk} - z\bbI)^{-1}\alpha_k \alpha_k^*(\bbB_{nk} - z\bbI)^{-1}}{\beta_k},  $$
	we have
	\begin{align} \label{7216}
	&\frac{1}{np} \sum_{k=1}^{n} \bbE  \dfrac{\bbr_k^* \bbQ^{-1} \bbT^\ell_n \bbD_{nk}^{-1} r_k}{\breve{\beta}_k}\\ \nonumber
	=& \frac{1}{np}\sum_{k=1}^{n} \left[ \bbE \dfrac{\bbr_k^* \bbQ^{-1} \bbT^\ell_n \bbD_n^{-1} r_k}{\breve{\beta}_k}       
	+\bbE \dfrac{\bbr_k^* \bbQ^{-1} \bbT^\ell_n \bbD_{nk}^{-1} \alpha_k \alpha_k^*\bbD_{nk}^{-1} \bbr_k}{n \breve{\beta}_k \beta_k} \right] \\ \nonumber
	=& \frac{1}{np}\sum_{k=1}^{n} \left[ \bbE \dfrac{\bbr_k^* \bbQ^{-1} \bbT^\ell_n \bbD_n^{-1} r_k}{\breve{\beta}_k}       
	+\bbE \dfrac{\bbr_k^* \bbQ^{-1} \bbT^\ell_n \bbD_{nk}^{-1} \alpha_k \alpha_k^*\bbD_{nk}^{-1} \bbr_k}{n \breve{\beta}_k^2 } \right] + o(1)	\\ \nonumber
	=& \frac{1}{np}\sum_{k=1}^{n} \left[ \bbE \dfrac{\bbr_k^* \bbQ^{-1} \bbT^\ell_n \bbD_n^{-1} r_k}{\breve{\beta}_k}       
	+\bbE \dfrac{\bbr_k^* \bbQ^{-1} \bbT^\ell_n \bbD_{nk}^{-1} (\bbr_k \bbr^*_k + \bbT_n)\bbD_{nk}^{-1} \bbr_k}{n \breve{\beta}_k^2 } \right] + o(1)	\\ \nonumber
	=& \frac{1}{np}\sum_{k=1}^{n} \left[ \bbE \dfrac{\bbr_k^* \bbQ^{-1} \bbT^\ell_n \bbD_n^{-1} r_k}{\breve{\beta}_k}       
	+\bbE \dfrac{\bbr_k^* \bbQ^{-1} \bbT^\ell_n \bbD_{nk}^{-1}\bbr_k \bbr^*_k\bbD_{nk}^{-1} \bbr_k}{n \breve{\beta}_k^2 } \right] + o(1)
	\end{align}
	Moving the second term to the left hand side, by noticing
	\begin{align*}
	1 - \bbE \dfrac{\bbr_k^* \bbD_{nk}^{-1} \bbr_k}{\breve{\beta}_k}
	= \bbE \dfrac{1+\frac{1}{n}\mathrm{tr} \bbT_n \bbD_{nk}^{-1}}{\breve{\beta}_k} 
	= &\bbE \dfrac{1+\frac{1}{n}\mathrm{tr} \bbT_n \bbD_n^{-1}}{n \breve{\beta}_k} + O(n^{-1}) \\
	=& \dfrac{1+y \bbE g_n(z)}{n \breve{\beta}_k} + O(n^{-1}),
	\end{align*}
	we obtain
	\begin{align}\label{7217}
	\frac{1}{np} \sum_{k=1}^{n} \bbE  \dfrac{\bbr_k^* \bbQ^{-1} \bbT^\ell_n \bbD_{nk}^{-1} \bbr_k}{\breve{\beta}_k} 
	=& \frac{1}{np} \sum_{k=1}^{n}  \bbE \dfrac{\bbr_k^* \bbQ^{-1} \bbT^\ell \bbD_n^{-1} \bbr_k}{1+y \bbE g_n(z)} + o(1) \\ \nonumber
	=& \frac{1}{np}\bbE \dfrac{\mathrm{tr} \bbR_n \bbR_n^* \bbQ^{-1} \bbT^\ell_n \bbD_n^{-1}}{1+y \bbE g_n(z)} + o(1).
	\end{align}
	Similarly, we have
	\begin{align}\label{7218}
	\frac{1}{np} \sum_{k=1}^{n} \bbE \dfrac{\mathrm{tr} [\bbT_n \bbQ^{-1} \bbT^\ell_n \bbD_{nk}^{-1}] }{\breve{\beta}_k} 
	=& \frac{1}{np} \sum_{k=1}^{n} \bbE \dfrac{\mathrm{tr} [\bbT_n \bbQ^{-1} \bbT^\ell_n \bbD_{nk}^{-1}] }{\beta_k} + O(n^{-1})\\ \nonumber
	=& -\frac{1}{p} \bbE z \mathrm{tr} [\bbT_n \bbQ^{-1} \bbT^\ell_n \bbD_n^{-1}] 
	\bbE \underline{s}_n(z) + O(n^{-1}) 
	\end{align}
	Substituting \eqref{7217} and \eqref{7218} into \eqref{726}, we obtain
	\begin{align*}
	&\frac{1}{p} \mathrm{tr} \bbT^\ell_n \bbQ^{-1} - \frac{1}{p}\bbE \mathrm{tr} \bbT^\ell_n\bbD_n^{-1}\\
	=&\frac{1}{np}\dfrac{\mathrm{tr}\bbR_n \bbR_n^*\bbQ^{-1}\bbT^\ell_n\bbD_n^{-1} }{1+y \bbE g_n(z)} +\frac{1}{p}\bbE z \mathrm{tr} [\bbT_n\bbQ^{-1}\bbT^\ell_n\bbD_n^{-1}] \bbE \underline{s}_n(z) + o(1) - \frac{1}{p} \bbE \bbD_n^{-1} K \bbQ^{-1}.
	\end{align*}
	Therefore, we will have
	\begin{align*}
	\frac{1}{p}  \mathrm{tr} \bbT^l_n \bbQ^{-1} - \frac{1}{p}\bbE \mathrm{tr} \bbT^l_n(\bbB_n-z\bbI)^{-1} \to 0, \ \mathrm{for} \ \ell=0,1.
	\end{align*}
	
	Examining the above proof, we find there are the following errors yield in the transformations:
	\begin{itemize}
		\item [(1)] In \eqref{726} due to the change $ \beta_k \to \breve{\beta}_k $, 
		\begin{align}\label{7469}
		e_1 = \dfrac{1}{p} \sum_{k=1}^{p} \bbE \alpha_k^* \bbQ^{-1} \bbT^{\ell} \bbD^{-1}_{nk} \alpha_k \left(  \dfrac{1}{\beta_k}- \dfrac{1}{\breve{\beta}_k} \right);
		\end{align}
		
		\item [(2)]  In \eqref{7216} again due to the change $ \beta_k \to \breve{\beta}_k $, 
		\begin{align}\label{7470}
		e_2 = \dfrac{1}{p} \sum_{k=1}^{p} \bbE \frac{ \bbr_k^* \bbQ^{-1} \bbT^{\ell} \bbD^{-1}_{nk} \alpha_k\alpha_k^*\bbD^{-1}_{nk} \bbr_k  }{n \breve{\beta}_k} \left(  \dfrac{1}{\beta_k}- \dfrac{1}{\breve{\beta}_k} \right);
		\end{align} 
		
		\item [(3)] And removing the term involved $ \bbT $
		\begin{align}\label{7471}
		e_3 = \dfrac{1}{np} \sum_{k=1}^{p} \bbE \frac{ \bbr_k^* \bbQ^{-1} \bbT^{\ell} \bbD^{-1}_{nk} \bbT \bbD^{-1}_{nk} \bbr_k  }{n \breve{\beta}_k};
		\end{align} 
		
		\item [(4)] In \eqref{7217} changing $ g_n(z) \to \bbE g_n(z) $ in the denominator
		\begin{align}\label{7472}
		e_4 = \dfrac{1}{np} \sum_{k=1}^{p} \bbE \frac{ \bbr_k^* \bbQ^{-1} \bbT^{\ell} \bbD^{-1}_{nk} \bbT \bbD^{-1}_{nk} \bbr_k  }{1 + y_n \bbE g_n } 
		\frac{ y_n(\bbE g_n - g_n)}{1 + y_n g_n};
		\end{align} 
		
		\item [(5)] In \eqref{7218} changing $ \breve{\beta}_k $ back to $ \beta_k $ 
		\begin{align}\label{7473}
		&e_5 = \dfrac{1}{np} \sum_{k=1}^{n} \bbE \frac{ \mathrm{tr} \bbT_n \bbQ^{-1} \bbT^{\ell}_n (\bbD^{-1}_{nk}-\bbD^{-1}_n) }{ \breve{\beta}_k};\\ \nonumber
		&e_6 = \dfrac{1}{np} \sum_{k=1}^{n} \bbE \mathrm{tr} [\bbT_n \bbQ^{-1} \bbT^{\ell}_n \bbD^{-1}_n ] \left(  \dfrac{1}{\breve{\beta}_k}- \dfrac{1}{\beta_k} \right);
		\end{align}
	\end{itemize}
	
	Before estimating the six errors, we point out that  
	$ \bbE \lvert \varepsilon_k \rvert ^2= O(n^{-1}) $ when $ z = x + i v_n $ with $ x \in [a, b] $, whose proof is the same as that for $ \eqref{7455}(b) $.  And the order $ O(n^{-1}) $ in the refined estimate is independent of $ v_n $, but $ x $ should be restricted in [a, b].
	And the estimates $ \left|  \dfrac{1}{\beta_k} \right|  $ and $ \left|  \dfrac{1}{\breve{\beta}_k} \right| $ are improved to be bounded in moment in Lemma \ref{lemma729}. 
	
	In Lemma \ref{lemma725}, it is proved that $ \bbQ^{-1} $ is bounded when $ x \in [a, b] $. We can get that $ \bbQ^{-1}$ is also bounded $ x \in [a', b'] $. 
	The estimate $ \bbE \lvert  \alpha_k^* \bbQ^{-1} \bbT_n^\ell \bbD_{nk}^{-1} \alpha_k \rvert^2=O(1)  $ remains unchanged as $ x \in [a, b] $ which can be proved by similar approach as showing $ \eqref{7455}(b) $. 
	
	Using the identity
	\begin{align}\label{7474}
	\dfrac{1}{\beta_k} - \dfrac{1}{\breve{\beta}_k} = -\dfrac{\varepsilon_k}{\breve{\beta}_k^2} + \dfrac{\varepsilon_k^2}{\breve{\beta}_k^2 \beta_k} .
	\end{align}
	The error term $ e_1 $ in $ \eqref{7469} $ is split into two terms:
	\begin{align}\label{7475}
	e_{11} = -\dfrac{1}{p} \sum_{k=1}^{p}  \bbE \frac{\varepsilon_k \alpha_k^* \bbQ^{-1} \bbT^{\ell}_n \bbD^{-1}_{nk} \alpha_k }{\breve{\beta}_k^2},
	\end{align}
	and
	\begin{align}\label{7476}
	e_{12} = \dfrac{1}{p} \sum_{k=1}^{p}  \bbE \frac{\varepsilon_k^2 \alpha_k^* \bbQ^{-1} \bbT^{\ell}_n \bbD^{-1}_{nk} \alpha_k }{\breve{\beta}_k^2 \beta_k},
	\end{align}
	
	Let $ \bbE_{(k)} $ denote the conditional expectation given all random vectors except
	$ \bbx_k $. 
	Then by Cauchy-Schwarz and Lemma \ref{lemma729}, when $ x \in [a, b] $, we obtain
	\begin{align*}
	\lvert e_{11} \rvert  & \leq  \dfrac{1}{p} \sum_{k=1}^{p} 
	\left|   \bbE \frac{\varepsilon_k(\alpha_k^*\bbQ^{-1} \bbT_n^{\ell} \bbD_{nk}^{-1} \alpha_k - \bbE_{(k)} \alpha_k^* \bbQ^{-1} \bbT_n^{\ell} \bbD_{nk}^{-1} \alpha_k )}{\breve{\beta}_k^2}    \right| \\
	& \leq \dfrac{K}{p} \sum_{k=1}^{p} \bbE \lvert  \varepsilon_k( \alpha_k^*\bbQ^{-1} \bbT_n^{\ell} \bbD_{nk}^{-1} \alpha_k - \bbE_{(k)} \alpha_k^* \bbQ^{-1} \bbT_n^{\ell} \bbD_{nk}^{-1} \alpha_k  )   \rvert \\
	& \leq  \dfrac{K}{p}   \left(  \sum_{k=1}^{p} \bbE \lvert \varepsilon_k \rvert^2  \sum_{k=1}^{p} \bbE  \lvert \alpha_k^*\bbQ^{-1} \bbT_n^{\ell} \bbD_{nk}^{-1} \alpha_k - \bbE_{(k)} \alpha_k^* \bbQ^{-1} \bbT_n^{\ell} \bbD_{nk}^{-1} \alpha_k  \rvert^2  \right) ^{1/2} \\
	& = O(n^{-1}),
	\end{align*}
	where the proof of the last step is similar to that of $ \eqref{7455}(b) $. 
	At the same time,
	\begin{align*}
	\lvert e_{12} \rvert  & \leq  \dfrac{1}{p}  \left|   \sum_{k=1}^{p} 
	\bbE \frac{\varepsilon_k^2 \alpha_k^*\bbQ^{-1} \bbT_n^{\ell} \bbD_{nk}^{-1} \bbT_n \bbD_{nk}^{-1} \alpha_k }{\breve{\beta}_k^2 \beta_k }    \right| \\
	& \leq \dfrac{C}{p} \sum_{k=1}^{p} \bbE \lvert  \varepsilon_k^2 
	\alpha_k^*\bbQ^{-1} \bbT_n^{\ell} \bbD_{nk}^{-1} \bbT_n \bbD_{nk}^{-1} \alpha_k    \rvert \\
	& \leq  \dfrac{C}{p} \sum_{k=1}^{p}  ( \bbE \lvert \varepsilon_k^2  \bbE_{(k)} \alpha_k^*\bbQ^{-1} \bbT_n^{\ell} \bbD_{nk}^{-1} \bbT_n \bbD_{nk}^{-1} \alpha_k   \rvert \\
	& \quad +  \bbE \lvert \varepsilon_k^2 ( \alpha_k^*\bbQ^{-1} \bbT_n^{\ell} \bbD_{nk}^{-1} \bbT_n \bbD_{nk}^{-1} \alpha_k   - \bbE_{(k)} \alpha_k^*\bbQ^{-1} \bbT_n^{\ell} \bbD_{nk}^{-1} \bbT_n \bbD_{nk}^{-1} \alpha_k )  \rvert   ) \\
	& = O(n^{-1}) + o(n^{-1}),
	\end{align*}
	where the estimates in the last step, the first term is similar to that of
	$ \eqref{7455}(b) $ and the second term follows by Cauchy-Schwarz and
	$$ \bbE \lvert \alpha_k^*\bbQ^{-1} \bbT_n^{\ell} \bbD_{nk}^{-1} \bbT_n \bbD_{nk}^{-1} \alpha_k   - \bbE_{(k)} \alpha_k^*\bbQ^{-1} \bbT_n^{\ell} \bbD_{nk}^{-1} \bbT_n \bbD_{nk}^{-1} \alpha_k \rvert^2 
	\leq  C n^{-1} v_n^{-2},$$
	and
	$$ \bbE \lvert \varepsilon_k^4 \rvert \leq C n^{-2} v_n^{-4}, $$
	provided $ v_n = n^{-\delta} $ with $ \delta < 1/3 $.
	
	To evaluate $ e_2 $, notice that
	$$  \bbE_{(k)} \alpha_k \alpha_k^* = \dfrac{1}{n} ( \bbr_k \bbr_k^* + T_n^{1/2}  \bbD^{-1}_{nk} \bbT^{1/2}_n ) , $$
	and consequently we have
	$$  \lvert e_2 \rvert  \leq \lvert e_{21} \rvert  + \lvert e_{22}  \rvert + \lvert e_{23}  \rvert , $$
	where
	\begin{align*}
	\lvert e_{21} \rvert = & \dfrac{1}{p} 
	\left|   \sum_{k=1}^{p} \bbE \frac{\bbr_k^*\bbQ^{-1} \bbT_n^{\ell} \bbD_{nk}^{-1}
		(\alpha_k \alpha_k^* - \bbE_{(k)} \alpha_k \alpha_k^*  ) \bbD_{nk}^{-1} \bbr_k }{n \breve{\beta}_k}  \left( \dfrac{1}{\beta_k} - \dfrac{1}{\breve{\beta}_k}  \right)  \right| \\
	\leq & \dfrac{C}{np} \left(  \sum_{k=1}^{p} \bbE \frac{ \lvert \bbr_k^*\bbQ^{-1} \bbT_n^{\ell} \bbD_{nk}^{-1} (\alpha_k \alpha_k^* - \bbE_{(k)} \alpha_k \alpha_k^*  ) \bbD_{nk}^{-1} \bbr_k \rvert^2 }{ \breve{\beta}_k^4} \sum_{k=1}^{p} \bbE \lvert \varepsilon_k \rvert^2  \right) ^{1/2} \\
	\leq  &  \dfrac{C}{n^2 p} 
	\left( \sum_{k=1}^{p} \bbE \breve{\beta}_k^{-4} \left( \lvert  \bbr_k^*\bbQ^{-1} \bbT^{\ell} \bbD_{nk}^{-1} \bbr_k \rvert^2 \|  \bbT_n^{1/2} \bbD_{nk}^{-1} \bbr_k  \|^2   \right. \right.\\ & \left. \left. 
	+  \bbE \|   \bbr_k^*\bbQ^{-1} \bbT^{\ell} \bbD_{nk}^{-1} \bbT_n^{1/2} \|^2 \lvert \bbr_k^* \bbD_{nk}^{-1} \bbr_k  \rvert^2   
	+  \bbE \|  \bbr_k^* \bbQ^{-1} \bbT_n^{\ell} \bbD_{nk}^{-1} \bbT_n^{1/2} \|^2    \|  \bbr_k^*\bbD_{nk}^{-1} \bbT_n^{1/2}  \|^2      \right. \right.\\ & \left. \left.    
	+ \bbE \lvert  \bbr_k^*\bbQ^{-1} \bbT_n^{\ell} \bbD_{nk}^{-1} \bbT_n \overline{\bbD}_{nk}^{-1} \bbr_k  \rvert^2  \right) \right) ^{1/2} 
	\end{align*}
	where we have used the fact that
	$$  \sup_{x\in [a,b]} \sum_{k=1}^{p} \bbE \lvert  \varepsilon_k \rvert^2= O(1).  $$
	By Lemma $ \ref{lemma730} $, we have
	\begin{align*}
	&\sum_{k=1}^{p} \bbE \breve{\beta}_k^{-4}  \lvert  \bbr_k^*\bbQ^{-1} \bbT^{\ell} \bbD_{nk}^{-1} \bbr_k  \rvert^2  \|  \bbT_n^{1/2} \bbD_{nk}^{-1} \bbr_k  \|^2
	\leq C n \sum_{k=1}^{p} \bbE \breve{\beta}_k^{-4} (\bbr^*_k \overline{\bbD}_{nk}^{-1} \bbD^{-1}_{nk} \bbr_k)^2 \\
	\stackrel{m}{=} & C n \sum_{k=1}^{p} \bbE (1 + y g_0(z))^{-4} (\bbr^*_k \overline{\bbD}_{n}^{-1} \bbD^{-1}_{n} \bbr_k)^2  
	\leq C n^2 \sum_{k=1}^{p} \bbE \bbr^*_k \overline{\bbD}_{n}^{-2} \bbD^{-2}_{n} \bbr_k \\
	=& C n^2 \mathrm{tr} \bbR \bbR^* \overline{\bbD}^{-1}_n \bbD^{-1}_n = O(n^4),
	\end{align*}
	where the first inequality follows by Cauchy-Schwarz with $ \| \bbr^*_k
	\bbQ^{-1} \bbT^{\ell}_n \|^2 = O(n )$ and $ \bbQ^{-1} $, $ \bbT_n $ are bounded in spectral norm; the second inequality employs the Cauchy-Schwarz again and 
	$ \|\bbr_k\|^2 = O(n) $; 
	and the last estimation follows by assumption $ F([a, b]) = o(v^4_n) $.
	
	To estimate the second term, let $ \bbv_k = \bbT^{\ell}_n \bbQ^{-1} \bbr_k $. 
	We have
	\begin{align*}
	&	\bbr^*_k \bbD^{-1}_{n} \overline{\bbD}_{n}^{-1}  \bbv_k \\
	=& \bbr^*_k \bbD^{-1}_{nk} \overline{\bbD}_{nk}^{-1}  \bbv_k  
	- \frac{\bbr_k^* \bbD_{nk}^{-1} \alpha_k  \alpha_k^* \bbD_{nk}^{-1} \overline{\bbD}_{nk}^{-1}  \bbv_k}{\beta_k} 
	- \frac{\bbr_k^* \bbD_{nk}^{-1} \overline{\bbD}_{nk}^{-1} \alpha_k  \alpha_k^* \overline{\bbD}_{nk}^{-1}  \bbv_k}{\overline{\beta}_k}  \\
	&+ \frac{\bbr_k^* \bbD_{nk}^{-1} \alpha_k  \alpha_k^* \bbD_{nk}^{-1} \overline{\bbD}_{nk}^{-1}  \alpha_k  \alpha_k^* \overline{\bbD}_{nk}^{-1}  \bbv_k}{\lvert \beta_k \rvert^2 } \\
	\stackrel{m}{=}&   \bbr^*_k \bbD^{-1}_{nk} \overline{\bbD}_{nk}^{-1}  \bbv_k
	- \frac{\bbr_k^* \bbD_{nk}^{-1} (\bbr_k  \bbr_k^* + \bbT_n ) \bbD_{nk}^{-1} \overline{\bbD}_{nk}^{-1}  \bbv_k}{n \breve{\beta}_k} 
	- \frac{\bbr_k^* \bbD_{nk}^{-1} \overline{\bbD}_{nk}^{-1} (\bbr_k  \bbr_k^* + \bbT_n ) \overline{\bbD}_{nk}^{-1}  \bbv_k}{n \overline{\breve{\beta}}_k}\\
	& + \frac{\bbr_k^* \bbD_{nk}^{-1} (\bbr_k  \bbr_k^* + \bbT_n )  \overline{\bbD}_{nk}^{-1}  (\bbr_k  \bbr_k^* + \bbT_n )  \bbD_{nk}^{-1} \overline{\bbD}_{nk}^{-1} \bbv_k}{n^2 \lvert \breve{\beta}_k\rvert^2} \\
	\stackrel{m}{=}& \bbr^*_k \bbD^{-1}_{nk} \overline{\bbD}_{nk}^{-1}  \bbv_k 
	-  \frac{\bbr_k^* \bbD_{nk}^{-1} \bbr_k  \bbr_k^*  \bbD_{nk}^{-1} \overline{\bbD}_{nk}^{-1}  \bbv_k}{n \breve{\beta}_k}  
	- \frac{\bbr_k^* \bbD_{nk}^{-1} \overline{\bbD}_{nk}^{-1} \bbr_k  \bbr_k^*  \overline{\bbD}_{nk}^{-1}  \bbv_k}{n \overline{\breve{\beta}}_k} \\
	& + \frac{\bbr_k^* \bbD_{nk}^{-1} \bbr_k  \bbr_k^* \bbD_{nk}^{-1} \overline{\bbD}_{nk}^{-1}  \bbr_k  \bbr_k^*  \overline{\bbD}_{nk}^{-1} \bbv_k}{n^2 \lvert \breve{\beta}_k\rvert^2}+ O(v_n^{-4})\\
	\stackrel{m}{=}& \left( \frac{1 + y g_0(z)}{\breve{\beta}_k} \bbr_k^* \bbD_{nk}^{-1} \overline{\bbD}_{nk}^{-1} \bbv_k  
	- \frac{\bbr_k^* \bbD_{n}^{-1} \overline{\bbD}_{n}^{-1} \bbr_k \bbr_k^* \overline{\bbD}_{n}^{-1} \bbv_k \lvert \breve{\beta}_k \rvert^2 }{n\lvert 1 + y g_0(z)\rvert^2(1 + y \overline{g}_0(z))}  \right. \\ & \left.
	+  \frac{\bbr_k^* \bbD_{n}^{-1} \bbr_k \bbr_k^* \bbD_{n}^{-1} \overline{\bbD}_{n}^{-1}  \bbr_k \bbr_k^* \overline{\bbD}_{n}^{-1} \bbv_k \lvert \breve{\beta}_k \rvert^2 }{n^2 \lvert 1 + y g_0(z)\rvert^4}  + O(v_n^{-4}) \right) 
	\end{align*}
	Smilarly, we have
	\begin{align*}
	& \bbv^*_k \bbD_{n}^{-1} \overline{\bbD}_{n}^{-1} \bbv_k 
	\stackrel{m}{=} \bbv^*_k \bbD_{nk}^{-1} \overline{\bbD}_{nk}^{-1} \bbv_k 
	- \frac{\bbv^*_k \bbD_{n}^{-1} \bbr_k \bbr_k^* \bbD_{nk}^{-1} \overline{\bbD}_{nk}^{-1} \bbv_k }{n(1 + y g_0(z))} \\
	& - \frac{\bbv^*_k \bbD_{nk}^{-1}  \overline{\bbD}_{nk}^{-1} \bbr_k \bbr_k^*\overline{\bbD}_{n}^{-1} \bbv_k }{n(1 + y \overline{g}_0(z))}  
	+ \frac{\bbv^*_k \bbD_{n}^{-1} \bbr_k \bbr_k^*  \bbD_{n}^{-1}   \overline{\bbD}_{n}^{-1} \bbr_k \bbr_k^*\overline{\bbD}_{n}^{-1} \bbv_k \lvert \breve{\beta}_k \rvert^2 }{n^2 \lvert 1 + y g_0(z)\rvert^4 }   .
	\end{align*}
	Therefore,
	\begin{align}\label{7477}
	&	\bbv^*_k \bbD_{nk}^{-1} \overline{\bbD}_{nk}^{-1} \bbv_k \\ \nonumber
	=& \bbv^*_k \bbD_{n}^{-1} \overline{\bbD}_{n}^{-1} \bbv_k 
	+ 2 \Re \left(   \frac{\bbv^*_k \bbD_{n}^{-1} \bbr_k \bbr_k^* \bbD_{n}^{-1} \overline{\bbD}_{nk}^{-1} \bbv_k \breve{\beta}_k }{n(1 + y g_0(z))^2 }  
	+ \frac{ \bbr^*_k \bbD_{n}^{-1} \overline{\bbD}_{n}^{-1} \bbr_k  \lvert  \bbr_k^* \bbD_{n}^{-1} \bbv_k \rvert^2 \lvert \breve{\beta}_k \rvert^2 \breve{\beta}_k }{n \lvert 1 + y g_0(z) \rvert^4 }  
	\right. \\ \nonumber & \left.
	+ \frac{ \bbr^*_k \bbD_{n}^{-1}  \bbr_k   \bbr_k^* \bbD_{n}^{-1}  \overline{\bbD}_{n}^{-1} \bbr_k   \bbr_k^* \overline{\bbD}_{n}^{-1}  \bbv_k  \lvert \breve{\beta}_k \rvert^2 \breve{\beta}_k }{n^2 \lvert 1 + y g_0(z) \rvert^4 (1 + y g_0(z))}  \right) 
	- \frac{ \lvert  \bbv^*_k \bbD_{n}^{-1} \bbr_k  \rvert^2 \bbr^*_k \bbD_{n}^{-1} \overline{\bbD}_{n}^{-1} \bbr_k  \lvert \breve{\beta}_k \rvert^2 }{n^2 \lvert 1 + y g_0(z) \rvert^4 }  
	\end{align}
	Therefore, the second term in $ e_{21} $ is dominated by
	$$ \dfrac{1}{n^2p} \left(  \sum_{k=1}^{p} \bbE \lvert \breve{\beta}_k \rvert^{-4}   \bbv^*_k \bbD_{nk}^{-1} \overline{\bbD}_{nk}^{-1} \bbv_k \lvert  \bbr^*_k \bbD_{nk}^{-1} \bbr_k \rvert^2  \right)^{1/2} = O(n^{-1} ) $$
	which can be derived by substituting $ \eqref{7477} $ and applying Lemma $ \ref{lemma730} $.
	\begin{align*}
	\lvert e_{22} \rvert  &= \dfrac{1}{np} \left|   \sum_{k=1}^{p} \bbE  \frac{\bbr_k^*\bbQ^{-1} \bbT^{\ell} \bbD_{nk}^{-1} \bbT \bbD_{nk}^{-1} \bbr_k }{n \breve{\beta}_k}  \left( \dfrac{1}{\beta_k} - \dfrac{1}{\breve{\beta}_k}  \right)   \right| \\
	& \leq \dfrac{C}{n p v_n^2} \sum_{k=1}^{p} \bbE \lvert \varepsilon_k \rvert 
	= O(n^{-3/2} v^{-4}_n) = O(n^{-1}). 
	\end{align*}
	\begin{align*}
	\lvert e_{23} \rvert  &= \dfrac{1}{np} \left|   \sum_{k=1}^{p} \bbE  \frac{\bbr_k^*\bbQ^{-1} \bbT^{\ell} \bbD_{nk}^{-1} \bbr_k \bbr_k^* \bbD_{nk}^{-1} \bbr_k }{n \breve{\beta}_k}  \left( \dfrac{1}{\beta_k} - \dfrac{1}{\breve{\beta}_k}  \right)   \right| \\
	&= \dfrac{1}{n^2 p} \left|   \sum_{k=1}^{p} \bbE  \frac{\bbr_k^*\bbQ^{-1} \bbT^{\ell} \bbD_{nk}^{-1} \bbr_k \bbr_k^* \bbD_{nk}^{-1} \bbr_k }{ \breve{\beta}_k}  
	\left( \dfrac{\varepsilon_k^2}{\breve{\beta}_k^3} -\dfrac{\varepsilon_k^3}{\breve{\beta}_k^3  \beta_k}  \right)   \right| \\
	& \leq  \dfrac{1}{n^2 p}   \sum_{k=1}^{p}  \frac{\bbr_k^*\bbQ^{-1} \bbT^{\ell} \bbD_{nk}^{-1} \bbr_k \bbr_k^* \bbD_{nk}^{-1} \bbr_k }{ \lvert \breve{\beta}_k \rvert^4}  \bbE \lvert \varepsilon_k \rvert^3  + O(n^{-3/2}) v_n^{-5}\\
	& \leq  \dfrac{C}{n^4 p}  \sum_{k=1}^{p} \bbE \frac{\bbr_k^*\bbQ^{-1} \bbT^{\ell} \bbD_{n}^{-1} \bbr_k \bbr_k^* \bbD_{n}^{-1} \bbr_k }{ \lvert \breve{\beta}_k \rvert^2 \lvert  1+y g_n(z) \rvert^2} \left(  \bbr_k^* \bbD_{nk}^{-1} \breve{\beta}_{nk}^{-1} \bbr_k + \mathrm{tr} \bbT_n^2   \right)  
	+ O(n^{-1}) \\
	&= O(n^{-1}).
	\end{align*}
	The estimation of $ e_3 $ and $ e_4 $ are as follows,
	$$  \lvert e_3 \rvert  \leq  \dfrac{1}{n^2 p} \left(  \sum_{k=1}^{p} \bbE \frac{\bbv_k^* \bbD_{nk}^{-1} \overline{\bbD}_{nk}^{-1} \bbv_k  }{\lvert \breve{\beta}_k\rvert^2  }    \sum_{k=1}^{p}  \frac{\bbr_k^* \overline{\bbD}_{nk}^{-1} \bbD_{nk}^{-1}  \bbr_k  }{\lvert \breve{\beta}_k\rvert^2  } \right)^{1/2}= O(n^{-1}).  $$
	and
	\begin{align*}
	&\lvert e_4 \rvert  =  \dfrac{1}{n p} \left|   \sum_{k=1}^{p} \bbE \frac{\bbr_k^* \bbQ^{-1} \bbT^{\ell}  \bbD_{n}^{-1} \bbr_k - \bbE \bbr_k^* \bbQ^{-1} \bbT^{\ell}  \bbD_{n}^{-1} \bbr_k }{1+y_n \bbE g_n  }  \frac{y_n (\bbE g_n -g_n) }{1+y_n g_n  }    \right|  \\
	\leq &\dfrac{C}{n p} \left(  \bbE \lvert  \mathrm{tr} \bbR_n \bbR_n^*\bbQ^{-1} \bbT^{\ell} \bbD_{n}^{-1} - \bbE \mathrm{tr}  \bbR_n \bbR_n^*\bbQ^{-1} \bbT^{\ell} \bbD_{n}^{-1}  \rvert^2 \bbE (g_n(z) - \bbE g_n(z))^2  \right)^{1/2} \\
	= & O(n^{-1}).
	\end{align*}
	In the proof for the last step, we have used conclusions $ \bbE \lvert g_n - \bbE g_n\rvert^2 = O(n^{-1}) $
	which is $ \eqref{7456}(b) $ with $ \ell= 1 $ and
	$$  \bbE \lvert \dfrac{1}{p} \mathrm{tr} (\bbR_n \bbR^*_n)/n \bbQ^{-1} \bbT^{\ell} \bbD_n^{-1}  - \bbE \mathrm{tr} (\bbR_n \bbR^*_n)/n \bbQ^{-1} \bbT^{\ell} \bbD_n^{-1} \rvert^2 = O(n^{-1})  $$
	which is the same as $ \eqref{7456}(b) $ when the matrix $ \bbT_n $ in $ \mathrm{tr}(\bbT_n \bbD_n^{-1} ) $ is replaced by $ (\bbR_n \bbR^*_n)/n \bbQ^{-1} \bbT^{\ell} $.
	
	The error $ e_5 $ can be estimated as
	\begin{align*}
	\lvert e_5 \rvert  &=\dfrac{1}{np} \left| 
	\sum_{k=1}^{p} \bbE\frac{ \mathrm{tr} [\bbT_n \bbQ^{-1} \bbT^{\ell}_n \{ \bbD^{-1}_{nk} \alpha_k \alpha_k^* \bbD^{-1}_{nk} \} ] }{\breve{\beta}_k \beta_k} \right| \\
	&\stackrel{m}{=}\dfrac{1}{n^2p}   \left| 
	\sum_{k=1}^{p} \bbE\frac{ \mathrm{tr} [\bbT_n \bbQ^{-1} \bbT^{\ell}_n \{ \bbD^{-1}_{nk} (\bbr_k \bbr_k^* + \bbT_n )\bbD^{-1}_{nk} \} ] }{\breve{\beta}_k^2 } \right| \\
	& \leq \dfrac{1}{n^2 p}  \sum_{k=1}^{p} \left|   (\bbr^*_k \bbD^{-1}_{nk}
	\overline{\bbD}^{-1}_{nk}  \bbr_k + \mathrm{tr}(\bbD_{nk}^{-1} \overline{\bbD}^{-1}_{nk} ))  \right| 
	= O(n^{-1}).
	\end{align*}
	
	As for the last error $ e_6 $, we have
	\begin{align*}
	\lvert e_6 \rvert  &=\dfrac{1}{np} \left| \sum_{k=1}^{p}  \bbE \mathrm{tr} [\bbT_n \bbQ^{-1} \bbT^{\ell}_n \bbD^{-1}_{n} ] \left( \dfrac{\varepsilon_k}{\breve{\beta}_k^2} -\dfrac{\varepsilon_k^2}{\breve{\beta}_k^2 \beta_k}  \right)  \right| \\
	&\leq \dfrac{C}{n p}  \sum_{k=1}^{p}  \left[  \left| 
	\bbE \mathrm{tr} [\bbT_n \bbQ^{-1} \bbT^{\ell}_n  ( \bbD^{-1}_{n}-\bbE \bbD_n^{-1} ) ]  \dfrac{\varepsilon_k}{\breve{\beta}_k^2 }  \right| + \bbE \lvert \mathrm{tr} [\bbT_n \bbQ^{-1} \bbT^{\ell}_n  \bbD^{-1}_{n}] \rvert \dfrac{\lvert \varepsilon_k \rvert^2}{ \lvert \breve{\beta}_k\rvert^2 }   \right] \\
	& \leq  O(n^{-1}),
	\end{align*}
	where the first term was estimated in $ e_4 $ and the second term can be estimated
	rountinely.
	
	Now, let $ s^0_n $ and $ g^0_n $ be the solutions to the system of equations
	$$ s^0_n= \int  \frac{\mathrm{d} H_n(u, t)}{\dfrac{u}{1+y_n g^0_n} - (1 + t y_n s^0_n)z + t(1 - y_n) }, $$
	$$ g^0_n  = \int  \frac{ t \mathrm{d} H_n(u, t)}{\dfrac{u}{1+y_n g^0_n} - (1 + t y_n s^0_n)z + t(1 - y_n) }.$$
	From these and $ \eqref{7468} $, we have
	\begin{align*}
	\lvert \bbE s_n - s^0_n \rvert  &= \left|  \int  \frac{\mathrm{d} H_n(u, t)}{\dfrac{u}{1+y_n \bbE g_n} - (1 + t y_n \bbE s_n)z + t(1 - y_n) } \right. \\ & \left.
	-\int  \frac{\mathrm{d} H_n(u, t)}{\dfrac{u}{1+y_n g^0_n} - (1 + t y_n s^0_n)z + t(1 - y_n) } + O(n^{-1})\right| \\
	& \leq  y_n \widetilde{A}_1 \lvert \bbE g_n - g^0_n \rvert  + y_n  \widetilde{B}_1 \lvert \bbE s_n - s^0_n \rvert  + O(n^{-1}),\\
	\lvert \bbE g_n - g^0_n \rvert  &= \left|  \int  \frac{ t \mathrm{d} H_n(u, t)}{\dfrac{u}{1+y_n \bbE g_n} - (1 + t y_n \bbE s_n)z + t(1 - y_n) } \right. \\ & \left.
	-\int  \frac{ t \mathrm{d} H_n(u, t)}{\dfrac{u}{1+y_n g^0_n} - (1 + t y_n s^0_n)z + t(1 - y_n) } + O(n^{-1})\right| \\
	& \leq  y_n \widetilde{A}_2 \lvert \bbE g_n - g^0_n \rvert  + y_n  \widetilde{B}_2 \lvert \bbE s_n - s^0_n \rvert  + O(n^{-1}),
	\end{align*}
	where
	\begin{align*}
	\widetilde{A}_j & = \int  \frac{  \frac{u t^{j-1}}{ \lvert 1+y_n \bbE g_n \rvert \lvert 1+y_n g^0_n \rvert }   \mathrm{d} H_n(u, t)}{\left| \dfrac{u}{1+y_n \bbE g_n}- (1 + y_n \bbE s_n)z + t (1 - y_n) \right|    \left| \dfrac{u}{1+y_n g^0_n} - (1 +  y_n s^0_n)z + t(1 - y_n)\right|  }  \\
	& \to  \widetilde{A}_{j0} = \int  \frac{  \frac{u t^{j-1}}{  \lvert 1+y_n g_0 \rvert^2 }   \mathrm{d} H_n(u, t)}{ \left| \dfrac{u}{1+yg_0} - (1 + y s_0)z + t(1 - y )\right|^2  } , j = 1, 2, \\
	\widetilde{B}_j & = \int  \frac{  t^j   \mathrm{d} H_n(u, t)}{\left| \dfrac{u}{1+y_n \bbE g_n}- (1 + y_n \bbE s_n)z + t (1 - y_n) \right|    \left| \dfrac{u}{1+y_n g^0_n} - (1 +  y_n s^0_n)z + t(1 - y_n)\right|  }  \\
	& \to  \widetilde{B}_{j0} = \int  \frac{  t^j \mathrm{d} H_n(u, t)}{ \left| \dfrac{u}{1+yg_0} - (1 +  y s_0)z + t(1 - y )\right|^2  } , j = 1, 2. 
	\end{align*}
	From these, one can easily derive that
	$$  \lvert \bbE s_n - s^0_n \rvert  \leq  (1 - y_n \widetilde{B}_1 - y^2_n \widetilde{A}_1 \widetilde{B}_2 (1 - y_n \widetilde{A}_2)^{-1})^{-1}O(n^{-1}) = O(n^{-1}).  $$
	In the above, the convergence of $ \widetilde{A}_j ( \widetilde{B}_j ) $ to $ \widetilde{A}_{j0}(\widetilde{B}_{j0}) $ follows by DCT and
	$ (1 - y_n \widetilde{B}_1 - y^2_n \widetilde{A}_1 \widetilde{B}_2 (1 - y_n \widetilde{A}_2)^{-1})^{-1} $ has an upper bound follows from the convergence and 
	$ y \widetilde{B}_{10} - y^2 \widetilde{A}_{10} \widetilde{B} _{20} (1 - y \widetilde{A}_{20})^{-1} < 1 $ which can be showing by comparing the imaginary part of \eqref{712.1}.
	
	Therefore, we can get for all $ n $ sufficiently large, 
	\begin{align*}
	\sup_{x\in[a,b]} \lvert \bbE s_n(x + i v_n) - s^0_n(x + i v_n)\rvert = O(n^{-1}). 
	\end{align*}
	which is \eqref{7467}.
	

	\section{Completing the proof}
	
	From the last two sections, combining $ \eqref{7455} $ (b) and $ \eqref{7467} $, for any $ \delta \in (0, 1/64) $, $ v_n = n^{-\delta} $  and $ x\in [a, b], $ we get 
	\begin{align}\label{7478}
	\sup_{x\in [a, b]} (n v_n) \lvert s_n (x + i v_n) - s^0_n(x + i v_n)\rvert = o(1) \quad \mathrm{a.s.}.
	\end{align}
	
	It is clear that \eqref{7478} is true when the imaginary part of $ z $ is replaced by a constant multiple
	of $ v_n $. So we have
	\begin{align*}
	\sup_{x\in [a, b]}  \lvert s_n (x + i\sqrt{k} v_n) - s^0_n(x + i\sqrt{k} v_n)\rvert = o(1/(n v_n)) \quad \mathrm{a.s.}.
	\end{align*}
	Taking the imaginary part and taking differences, we get
	\begin{align*}
	&\sup_{x\in[a,b]} \left| \int \dfrac{\mathrm{d}(F^{\bbB_n}(\lambda) - F^{y_n,H_n}(\lambda))}{	((x - \lambda)^2 + v_n^2 )((x - \lambda)^2 + 2v_n^2) \cdots ((x - \lambda)^2 + pv_n^2)}\right|=o(1),  \ \mathrm{a.s.} 
	\end{align*} 
	We split up the integral and get
	\begin{align}\label{6251}
	&\sup_{x\in[a,b]} \left| \int \dfrac{I([a^{'}, b^{'}]^c)\mathrm{d}(F^{\bbB_n}(\lambda) - F^{y_n,H_n}(\lambda))}{	((x - \lambda)^2 + v_n^2 )((x - \lambda)^2 + 2v_n^2) \cdots ((x - \lambda)^2 + pv_n^2)} \right. \\ \nonumber & \left.
	+ \sum_{ \lambda_j \in [a^{'}, b^{'}] }\dfrac{v_n^{2p}}{ ((x - \lambda)^2 + v_n^2 )((x - \lambda)^2 + 2v_n^2) \cdots ((x - \lambda)^2 + pv_n^2)}     \right|=o(1)	,  \ \mathrm{a.s.} 
	\end{align}  
	
	Now if for each term in a subsequence satisfying \eqref{6251}, there is at least
	one eigenvalue contained in $ [a,b] $, then the sum in \eqref{6251} will be uniformly bounded away from 0. Thus, the integral in \eqref{6251} must also stay uniformly bounded away from $ 0 $. 
	But the integral converges to 0 a.s. since the integrand is bounded
	and, with probability one, both $ F^{\bbB_n} $ and $ F^{y_n,H_n} $ converge weakly to the same
	limit having no mass on $ [a^{'}, b^{'}] $. Thus, with probability one, no eigenvalues of
	$ \bbB_n $ will appear in $ [a,b] $ for all $ n $ sufficiently large. This completes the proof of
	Theorem \ref{th7221}.
	
	\section{Mathematical tools}
	
	
	\begin{lemma} (Theorem B.14 of \cite{bai2010spectral})\label{lemma3.1}
		Let $ F $ be a distribution function and let $ G $ be a function of bounded variation satisfying $ \int  \left|  F(x)-G(x) \right| \mathrm{d} x < \infty $. 
		Denote their Stieltjes transforms by $ f(z) $ and $ g(z) $, respectively. 
		Then, we have 
		\begin{align}\label{B29}
		\lVert F - G \rVert \leq  & \dfrac{1}{\pi(1-\kappa)(2\gamma - 1)}  \left[ \int_{-A}^{A}   \lvert f(z) - g(z)\rvert \mathrm{d} u  + 2 \pi v^{-1}  \int_{\lvert x \rvert >B} \lvert F(x)-G(x) \rvert \mathrm{d}x 
		\right.  \\ \nonumber    & +  \left.
		v^{-1} \sup_x  \int_{\lvert y \rvert \leq 2va} \lvert G(x + y) - G(x) \rvert \mathrm{d} y   \right],
		\end{align}
		where $ A $ and $ B $ are positive constants such that $ A > B $ and
		\begin{align}\label{B210}
		\kappa= \dfrac{4B}{\pi(A-B)(2\gamma-1)}  < 1 . 
		\end{align} 
	\end{lemma}
	
	\begin{lemma}( Lemma 3.3 of \cite{dozier2007empirical}) \label{lemma3.2}
		Let $ z \in \mathbb{C}^+  $ with $ v = \Im z $, $ \bbA $ and $ \bbB $ $ n \times n  $ with $ \bbB $ Hermitian, and $ \bbr \in \mathbb{C}^n $. Then $$ \lvert  \mathrm{tr}((\bbB- z\bbI)^{-1} - (\bbB+ \bbr\bbr^*-z\bbI)^{-1})\bbA \rvert =\left|   \frac{\bbr^*(\bbB - z\bbI)^{-1}\bbA(\bbB - z\bbI)^{-1}\bbr}{1 + \bbr^*(\bbB - z\bbI)^{-1 }\bbr }  \right| \leq \dfrac{\lVert \bbA \rVert }{v}.$$
	\end{lemma}
	
	\begin{lemma}(Lemma 2.12 of  \cite{bai2010spectral})
		Let $ \{X_k\} $ be a complex martingale difference sequence with respect to the increasing 
		$ \sigma-field $ $ {\mathcal{F}_k} $. Then, for $ p > 1 $, 
		$$   \bbE \left|   X_k\right|^p \leq K_p \bbE \left(  \sum |X_k|^2 \right)^{p/2}. $$
	\end{lemma}
	
	\begin{lemma}(Lemma 2.13 of \cite{bai2010spectral}) \label{lemma2.13}
		Let $ \{X_k\} $ be a complex martingale difference sequence with respect to the increasing 
		$ \sigma-field $ $ {\mathcal{F}_k} $, and let $ \bbE_k $ denote conditional expectation
		$ w.r.t. \mathcal{F}_k. $ 
		Then, for $ p \geq 2 $, 
		$$\bbE \left|   X_k\right|^p \leq K_p \left(  \bbE \left(  \sum \bbE_{k-1} |X_k|^2  \right) ^{p/2} + \bbE \sum |X_k|^p  \right) $$
	\end{lemma}

	\begin{lemma}(Lemma B.26 of \cite{bai2010spectral}) \label{lemmab26}
		Let $ \bbA=(a_{ij}) $ be an $ n \times n $ nonrandom matrix and $ \bbX=(x_1, \dots, x_n)^{'} $
		be a random vector of independent entries. 
		Assume that $ \bbE x_i = 0, \bbE|x_i|^2 = 1,  $ and $ \bbE|x_j |^{\ell} \leq \nu_{\ell} $. 
		Then, for any $ p \geq 1 $, 
		$$ \bbE \vert \bbX^* \bbA \bbX - \mathrm{tr} \bbA \vert^p \leq C_p \left( (\nu_4 \mathrm{tr}(\bbA\bbA^*))^{p/2} + \nu_{2p} \mathrm{tr}(\bbA \bbA^*)^{p/2} \right),  $$
		where $ C_p $ is a constant depending on $ p $ only.
	\end{lemma}

	\begin{lemma}(Interlacing theorem) \label{p528inth}
		If $ \bbC $ is an $ (n - 1) \times (n - 1) $ major sub-matrix of the $ n \times n $ Hermitian matrix $ \bbA $, then $ \lambda_1(\bbA) \geq \lambda_1(\bbC) \geq \lambda_2(\bbA) \geq \cdots \geq \lambda_{n-1}(\bbC) \geq \lambda_n(\bbA) $, where $ \lambda_i(\bbA) $ denotes the $ i $-th largest eigenvalues of the Hermitian matrix $ \bbA $.
	\end{lemma}	
	
	\begin{lemma}(Corollary 7.3.8 of \cite{hom1985matrix}) \label{lemma2.5}
		For $ n \times p $ matrices $ \bbA $ and $ \bbB $ with respective singular values $ s_1(\bbA) \geq s_2(\bbA) \geq \cdots \geq s_q(\bbA) $, $ s_1(\bbB) \geq s_2(\bbB) \geq \cdots \geq s_q(\bbB) $, where
		$ q = \min(n, p) $, for all $ k = 1, 2, \cdots, q $, we have
		$$  \lvert s_k(\bbA) - s_k(\bbB) \rvert  \leq \|\bbA - \bbB \|.  $$	
	\end{lemma}

	\begin{acks}[Acknowledgments]
		Z. D. Bai was partially supported by NSFC Grant 12171198 and Team Project of Jilin Provincial Department of Science and Technology (No.20210101147JC).  J. Hu was supported by NSFC (Nos. 12171078, 11971097).
	\end{acks}

	\bibliographystyle{imsart-number} 

\end{document}